\documentclass[11pt]{amsart}

\usepackage[T1]{fontenc}
\usepackage[utf8]{inputenc}
\usepackage{lmodern}
\usepackage{microtype}
\usepackage{amsmath,amssymb,amsthm,mathtools}
\usepackage{mathrsfs}
\usepackage{graphicx}
\usepackage{float}
\usepackage{fancyhdr}
\usepackage{multirow}
\usepackage{multicol}
\usepackage{enumitem}
\usepackage{empheq}
\usepackage{cases}
\usepackage[all]{xy}
\usepackage[draft=false,setpagesize=false,pdfstartview=FitH,
colorlinks=true,citecolor=blue]{hyperref}

\numberwithin{equation}{section}

\newtheorem{theorem}{Theorem}[section]
\newtheorem{corollary}[theorem]{Corollary}
\newtheorem{proposition}[theorem]{Proposition}
\newtheorem{lemma}[theorem]{Lemma}

\theoremstyle{definition}

\theoremstyle{remark}
\newtheorem{remark}[theorem]{Remark}

\newcommand{\calW}{\mathcal W}
\newcommand{\Lam}{\Lambda}
\newcommand{\dd}{\,d}

\newcommand{\Ric}{\operatorname{Ric}}
\newcommand{\Id}{\operatorname{Id}}

\newcommand{\ip}[2]{\left\langle #1,#2\right\rangle}
\newcommand{\norm}[1]{\left|#1\right|}
\newcommand{\Rm}{\operatorname{Rm}}
\newcommand{\CP}{\mathbb{CP}}

\newcommand\restr[2]{{
  \left.\kern-\nulldelimiterspace
  #1
  \vphantom{\big|}
  \right|_{#2} 
  }}

\title[Conformal K\"ahler Rigidity of Einstein Four-Manifolds]{Conformal K\"ahler Rigidity of Einstein Four-Manifolds}

\author[G. Catino]{Giovanni Catino}
\address{Dipartimento di Matematica, Politecnico di Milano, Italy}
\email{giovanni.catino@polimi.it}

\author[D. Dameno]{Davide Dameno}
\address{Dipartimento di Matematica ``Giuseppe Peano'', Universit\`a degli Studi di Torino, Italy}
\email{davide.dameno@unito.it}

\subjclass[2020]{Primary 53C25, 53C24; Secondary 53C18, 53C21, 53C55}
\keywords{Einstein four-manifolds, harmonic self-dual Weyl tensor, conformally K\"ahler metrics, gravitational instantons, K\"ahler--Einstein surfaces}

\begin{document}

\begin{abstract}
For a compact, connected, oriented Einstein four-manifold, we prove that, if the largest eigenvalue of the self-dual Weyl curvature $W^+$ is everywhere simple, then, after at worst passing to a double cover, the metric is conformally K\"ahler with positive scalar curvature; more generally, this result holds for metrics with harmonic self-dual Weyl curvature. For Einstein metrics satisfying a uniform simplicity hypothesis on the largest eigenvalue, we further prove that either $W^+\equiv 0$, or $W^+$ nowhere vanishes and the previous conclusion holds. We also obtain extensions to complete Ricci-flat four-manifolds and an optimal pinching theorem for the holomorphic sectional curvature of compact K\"ahler--Einstein surfaces. The proof combines LeBrun's conformal normalization with the resulting weighted divergence equation and new first-order identities. A zero-capacity argument allows this method to be used across the zero set of $W^+$ and at infinity in the noncompact case. Finally, K3 surfaces and multicentered Gibbons--Hawking gravitational instantons show that our assumptions are sharp. 
\end{abstract}

\maketitle

\section{Introduction}

Let $(M,h)$ be an oriented \emph{Einstein} manifold of dimension four, i.e., its Ricci tensor $\Ric_h$ is proportional to the metric $h$: in this case, the curvature information, apart from the value of the scalar curvature, is all encoded in the \emph{self-dual} and \emph{anti-self-dual Weyl tensors}, which can be regarded as linear, symmetric, trace-free
endomorphisms 
\[
W_h^{\pm}:\Lam^\pm\longrightarrow\Lam^\pm
\]
of the bundles $\Lam^\pm$, arising as eigenspaces of the Hodge operator acting on the bundle $\Lam^2$ of 2-forms. 
Throughout the paper, we assume that every manifold we consider is connected. 

We focus on the symmetric operator $W_h^+$ and we denote its
eigenvalues by
\[
\alpha_h\geq\beta_h\geq\gamma_h, \quad \alpha_h+\beta_h+\gamma_h=0.
\]
The main aim of this paper is to show that a rather weak assumption on the two largest eigenvalues $\alpha_h$ and $\beta_h$ actually implies strong rigidity for $(M,h)$, connecting the
conformal properties of $h$ with the existence of K\"{a}hler structures on $M$. More precisely, we show that, after possibly passing to a double cover of $M$ (see Section \ref{sec:confchange} for a more complete explanation), 
\[
\alpha_h>\beta_h \quad\text{on }M \quad\Longrightarrow \quad
\beta_h=\gamma_h= -\dfrac{1}{2}\alpha_h \quad\text{on }M.
\]
A well-known result due
to Derdzi\'nski \cite[Proposition 5]{Derdzinski} implies that the conformal metric $g=\alpha_h^{2/3}h$ is 
\emph{K\"{a}hler with positive scalar curvature} and, as a consequence, $h$ has to be a Hermitian Einstein metric with positive scalar curvature.  

More precisely, we prove the following.
\begin{theorem}\label{thm:Einstein-simple-top}
Let $(M,h)$ be a compact oriented Einstein four-manifold. Assume that
\[
\alpha_h>\beta_h 
\quad\text{on }M.
\]
Then, after at worst passing to a double cover of $M$, $h$ is conformal to a K\"ahler metric with positive scalar curvature. In particular, $h$ is a Hermitian Einstein metric with positive scalar curvature.
\end{theorem}
This result gains importance when viewed in the context of classifying compact Einstein four-manifolds with positive scalar curvature: indeed, LeBrun \cite{LeBrun2012,LeBrun2013} proved that, in the Hermitian, simply connected setting, $(M,h)$ is either K\"ahler--Einstein or
isometric to the first del Pezzo surface
$\mathbb{CP}^2\#\overline{\mathbb{CP}^2}$ endowed with the Page
metric \cite{Page}, or to the second del Pezzo surface
$\mathbb{CP}^2\#2\overline{\mathbb{CP}^2}$ endowed with the
Chen--LeBrun--Weber metric \cite{ChenLeBrunWeber}. In the
		K\"{a}hler--Einstein case, by the classification and existence results of Tian \cite{Tian} and
		Odaka--Spotti--Sun \cite{OdakaSpottiSun}, the underlying complex surface $(M,J)$ is biholomorphic to
		$\mathbb{CP}^2$, $\mathbb{CP}^1\times\mathbb{CP}^1$ or
		$\mathbb{CP}^2\#k\overline{\mathbb{CP}^2}$, $3\leq k\leq8$, with
		the blown-up points in general position. As a consequence, Theorem \ref{thm:Einstein-simple-top} provides the classification of all compact, simply connected Einstein four-manifolds $(M,h)$ with positive scalar curvature and simple largest eigenvalue of $W_h^+$ (see Corollary \ref{cor:einstein-classification} for a more general result).

        We point out that, in this context, the proof of Theorem \ref{thm:Einstein-simple-top} only relies on $h$ having \emph{harmonic self-dual Weyl curvature}, i.e. 
        \[
        \delta_hW_h^+\equiv 0,
        \]
        where $\delta_h$ denotes the divergence operator: this condition is satisfied by all Einstein four-manifolds \cite{Besse}. Theorem \ref{thm:Einstein-simple-top} is a direct consequence of the following result.

        \begin{theorem}\label{thm:half-harmonic-simple-top}
Let $(M,h)$ be a compact oriented Riemannian four-manifold with
\[
    \delta_hW_h^+=0\quad\text{on }M.
\]
Assume that
\[
    \alpha_h>\beta_h
    \quad\text{on }M.
\]
Then, after at worst passing to a double cover of $M$, $h$ is conformal to a K\"ahler metric with positive scalar curvature.

Conversely, if $(M,g,J)$ is a compact K\"ahler surface, with respect to the complex orientation, and with $s_g>0$,
then $h=s_g^{-2}g$ satisfies $\delta_hW_h^+=0$ and
$\alpha_h>\beta_h$.
\end{theorem}
Given a Riemannian four-manifold $(M,h)$, on the set $\{W_h^+\neq0\}$ we can define the conformally invariant ratio
\[
    \rho=\frac{\beta_h}{\alpha_h}:
\]
this is a continuous function which satisfies
\[
    -\frac12\leq\rho\leq1,
\]
since $W_h^+$ is trace-free. 
One can immediately observe that
\begin{align*}
    \rho=-\frac12
    &\quad\Longleftrightarrow\quad
    \beta_h=\gamma_h=-\frac12\alpha_h,\\
    \rho=1
    &\quad\Longleftrightarrow\quad
    \alpha_h=\beta_h,\quad \gamma_h=-2\alpha_h.
\end{align*}
A simple computation shows that 
\begin{equation*}
    \frac{\det(W_h^+)}{|W_h^+|^3}
    =
    -\frac{\rho(1+\rho)}
    {2^{3/2}(1+\rho+\rho^2)^{3/2}}:
\end{equation*}
the right-hand side is non-increasing on $[-1/2,1]$ and attains its
minimum at $\rho=1$, i.e.
\[
    \frac{\det(W_h^+)}{|W_h^+|^3}
    =-\frac1{3\sqrt6}
    \quad\Longleftrightarrow\quad
    \alpha_h=\beta_h>0.
\]
With this in mind,
Theorem \ref{thm:half-harmonic-simple-top} can be rephrased as
\begin{equation}\label{eq:intro-rho-implication}
    \rho<1 \quad\text{on }M
    \quad\Longrightarrow\quad
    \rho=-\frac12\quad\text{on }M.
\end{equation}
This result improves upon the previously known results: indeed, in his seminal paper, Derdzi\'nski proved the existence of a conformal K\"ahler metric on $\{W_h^+\neq 0\}$ under the hypothesis $\#\operatorname{spec}(W_h^+)\leq 2$, i.e. $\rho\in\{-1/2,1\}$ \cite{Derdzinski}. More recently, Wu improved this result by assuming that $\det(W_h^+)>0$, i.e. $\rho<0$ \cite{Wu}, and LeBrun showed that it is sufficient to assume 
\[
\dfrac{\det(W_h^+)}{\norm{W_h^+}^3}\geq -\dfrac{5\sqrt{2}}{21\sqrt{21}},
\]
i.e. $\rho\leq 1/4$ \cite{LeBrun2021}. By \eqref{eq:intro-rho-implication} and the previous considerations, Theorem \ref{thm:half-harmonic-simple-top} gives the
following consequence.

\begin{corollary}\label{cor:det-version-hhw}
Let $(M,h)$ be a compact oriented Riemannian four-manifold with
$\delta_hW_h^+=0$. Assume that
\[
    \det(W_h^+)>-\frac1{3\sqrt6}|W_h^+|^3
    \quad\text{on }M.
\]
Then, after at worst passing to a double cover of $M$, $h$ is conformal to a K\"ahler metric with positive scalar curvature. 
\end{corollary}
The proofs build upon LeBrun's conformal method \cite{LeBrun2021}: indeed, we make use of the fact that, up to passing to the double cover of $M$ determined by the eigenspace of $W_h^+$ associated to $\alpha_h$, the smooth, real line subbundle $L\subset \Lam^+$ defined by this eigenspace is trivial and, therefore, we can exploit LeBrun's conformal normalization
\[
g=\alpha_h^{2/3}h.
\]
Together with the weighted version of the well-known Weitzenb\"ock formula for metrics with harmonic self-dual Weyl curvature discovered by Derdzi\'nski, we obtain the crucial integral identity \eqref{eq:domain-identity}. The new ingredient is a system of first-order identities for the eigenvalues of $W_g^+$ on its regular spectral set (Proposition \ref{prop:weighted-derdzinski}); these equations imply \eqref{eq:domain-identity} and therefore force an eigenform $\omega_1$, associated to $\alpha_g$, to be parallel and hence to be the K\"ahler form of $g$.
We point out that this approach comes from merging the ideas of 
Derdzi\'nski \cite{Derdzinski} and LeBrun \cite{LeBrun2021}: indeed, we make use of Derdzi\'nski's formulas for the derivatives of the eigenvalues of $W_h^+$, applied to the weighted self-dual Weyl curvature $\alpha_h^{-1/3}W_g^+$.

For Einstein metrics, we may allow $W_h^+$ to vanish \emph{a priori} by replacing the pointwise simplicity hypothesis with a uniform simplicity assumption on the largest eigenvalue. More precisely, the following results extend Theorem \ref{thm:half-harmonic-simple-top} and Corollary \ref{cor:det-version-hhw}:

\begin{theorem}\label{thm:main}
Let $(M,h)$ be a compact oriented Einstein four-manifold. Assume
that $W_h^+\not\equiv0$ and that, for some constant $\rho_0<1$,
\[
    \beta_h\leq\rho_0\alpha_h
    \quad\text{on }M.
\]
Then, after at worst passing to a double cover of $M$, $h$ is conformal to a K\"ahler metric with positive scalar curvature. In particular, $h$ is a Hermitian Einstein metric with positive scalar curvature.
\end{theorem}
Since $M$ is compact, obviously, if $\alpha_h>\beta_h$ everywhere, the eigenvalues satisfy the hypothesis of Theorem \ref{thm:main}. We can rephrase this result in terms of $\det(W_h^+)$ as well:
\begin{corollary}\label{cor:det-version}
Let $(M,h)$ be a compact oriented Einstein four-manifold. Assume that
$W_h^+\not\equiv0$ and that
\[
    \det(W_h^+)\geq-c|W_h^+|^3
    \quad\text{on }M
\]
for some constant
\[
    0\leq c<\frac1{3\sqrt6}.
\]
Then, after at worst passing to a double cover of $M$, $h$ is conformal to a K\"ahler metric with positive scalar curvature. In particular, $h$ is a Hermitian Einstein metric with positive scalar curvature.
\end{corollary}

Combining Theorem \ref{thm:main} with the classification of anti-self-dual Einstein four-manifolds with positive scalar curvature \cite{FriedrichKurke, HitchinKahler} and the classification of Hermitian Einstein four-manifolds cited above, we can classify all compact Einstein manifolds with positive scalar curvature and uniformly simple largest eigenvalue.

\begin{corollary}\label{cor:einstein-classification}
Let $(M,h)$ be a compact simply connected oriented Einstein
four-manifold with positive scalar curvature. Assume that 
\[
    \beta_h\leq\rho_0\alpha_h
    \quad\text{on }M
\]
for some $\rho_0<1$. Then, one of the following holds:
\begin{itemize}
\item $(M,h)$ is, up to homothety and isometry, either $\mathbb{S}^4$ or $\overline{\mathbb{CP}^2}$, endowed with their standard metrics;
\item $h$ is K\"ahler--Einstein, and the underlying complex surface is
biholomorphic to
\[
    \mathbb{CP}^2,
    \quad
    \mathbb{CP}^1\times\mathbb{CP}^1,
    \quad\text{or}\quad
    \mathbb{CP}^2\#k\overline{\mathbb{CP}^2},
    \quad 3\leq k\leq8,
\]
with the blown-up points in general position;
\item $h$ is Hermitian non-K\"ahler and is, up to homothety and isometry, the Page
metric on $\mathbb{CP}^2\#\overline{\mathbb{CP}^2}$ or the
Chen--LeBrun--Weber metric on
$\mathbb{CP}^2\#2\overline{\mathbb{CP}^2}$.
\end{itemize}
\end{corollary}	
The idea to prove these results is to apply the same computations used in Theorem \ref{thm:half-harmonic-simple-top}, combined with a zero-capacity argument for the zero set of $W_h^+$. More precisely, we use the Einstein condition to show that $W_h^+$ is a non-trivial solution of a system of first-order elliptic PDEs and we exploit a known result due to B\"ar to control the size of the zero set of $W_h^+$ \cite{Baer1999}: then, we construct a suitable sequence of Lipschitz functions $\chi_j$, which are compactly supported in $\{W_h^+\neq 0\}$ and allow us to deal with the boundary terms arising in the integration by parts procedure in $\{W_h^+\neq 0\}$, thanks to which we obtain again the identity \eqref{eq:domain-identity} on this open subset. At this point, it is sufficient to use Derdzi\'nski's celebrated result on the zeros of $W_h^+$ on Einstein four-manifolds to show that, if $W_h^+\not\equiv 0$, then $W_h^+$ nowhere vanishes and this concludes the proof \cite[Proposition 5 (iv)]{Derdzinski}. 

With these results in mind, taking the contrapositive of Theorem \ref{thm:main}, we can obtain a useful
restriction on Einstein metrics of nonpositive scalar curvature.

\begin{corollary}\label{cor:nonpositive-scalar-nonexistence}
Let $(M,h)$ be a compact oriented Einstein four-manifold with
nonpositive scalar curvature. If $W_h^+\not\equiv0$, then
\[
    \sup_{\{W_h^+\neq0\}}\frac{\beta_h}{\alpha_h}=1.
\]
\end{corollary}
We highlight the fact that the spectral hypotheses in Theorems \ref{thm:Einstein-simple-top} and \ref{thm:main} are sharp: indeed, we manage to show that some K3 surfaces, endowed with a Calabi--Yau metric, admit points at which, after
reversing the hyperk\"{a}hler orientation, we have
\[
    \alpha_h=\beta_h>0;
\] 
since these spaces do not admit metrics with positive scalar
curvature, we obtain that the conclusion of Theorem
\ref{thm:main} fails.

The same method can be adapted to study compact K\"{a}hler--Einstein surfaces $(M,g,J)$ with nonpositive scalar curvature. In this setting, we can analyze the holomorphic sectional curvature
$H$ and obtain a sharp classification result, under the same spectral hypothesis as in Theorem \ref{thm:main}: more precisely, if we denote by $H_{\min}$ and $H_{\max}$ the minimum and the maximum of the holomorphic sectional curvature, respectively, and by $H_{\mathrm{av}}$ its average over $\mathbb{CP}^1$, we obtain the following result.

\begin{theorem}\label{thm:siu-yang-pinching}
Let $(M,h,J)$ be a compact K\"ahler--Einstein surface with
nonpositive scalar curvature. Assume that there exists a constant
$\theta<2/3$ such that, at every point where
$H_{\max}>H_{\min}$,
\[
    H_{\mathrm{av}}-H_{\min}
    \leq
    \theta\bigl(H_{\max}-H_{\min}\bigr).
\]
Then $W_h^-\equiv 0$. Consequently, either $h$ is flat and $M$ is finitely
covered by a flat complex torus, or $h$ has constant negative
holomorphic sectional curvature and $M$ is a compact quotient of the
complex ball.
\end{theorem}
At every point $p\in M$ such that 
$H_{\max}>H_{\min}$, the ratio 
\[
\Theta:=\dfrac{H_{\mathrm{av}}-H_{\min}}{H_{\max}-H_{\min}}
\]
lies in $[1/3,2/3]$: thus, Theorem \ref{thm:siu-yang-pinching} only assumes a uniform gap below the upper bound $2/3$. Siu and Yang obtained the same rigidity theorem under a nonpositive
bisectional curvature assumption and imposing $\Theta<\frac{2}{3(1+\sqrt{6/11})}$ \cite{SiuYang}, while Guan proved the same conclusion
under the condition $\Theta\leq1/2$ without a sign assumption on the
bisectional curvature \cite{Guan2017}. 
Our statement requires no sign assumption on the sectional or the bisectional
curvature and
covers every uniform bound strictly below
$2/3$. We point out that Theorem \ref{thm:siu-yang-pinching} is optimal: indeed, recently, compact negatively curved K\"ahler--Einstein manifolds which are not covered by the complex ball have been constructed \emph{via} gluing and deformation methods \cite{GuenanciaHamenstadt}, which means that
the value $2/3$ cannot be included.

The combination of Derdzi\'nski's Weitzenb\"ock formula, LeBrun's conformal method and the zero-capacity argument (Proposition \ref{prop:capacity-form}) can be, to some extent, also applied to noncompact Einstein four-manifolds: in particular, in the presence of suitable volume growth and curvature decay assumptions at infinity, we can adapt our line of reasoning to the complete case. It is well known that the study of the eigenvalues of $W_h^+$ is closely related to the classification of the so-called \emph{gravitational instantons}, i.e. complete, noncompact four-manifolds admitting a Ricci-flat metric such that the $L^2$-norm of the curvature tensor is finite: indeed, these spaces can be distinguished into three different types, which correspond to different multiplicities of the eigenvalues of $W_h^+$
(see e.g. the recent survey \cite{LiSun} and the references therein).  
In our setting, we are able to prove the following general result for noncompact Ricci-flat manifolds:
\begin{theorem}\label{thm:intro-extensions}
Let $(M,h)$ be a complete, oriented, noncompact Ricci-flat
four-manifold. Assume that $W_h^+\not\equiv0$ and that, for some
$\rho_0<1$,
\begin{equation}\label{eq:noncompact-uniform-gap}
\beta_h\leq\rho_0\alpha_h
\quad\text{on }M.
\end{equation}
Suppose moreover that there exist $\nu,q\geq0$ and $p_0\in M$ such that
\begin{equation}\label{eq:noncompact-volume-curvature}
\begin{aligned}
\begin{cases}\operatorname{Vol}_h(B_h(p_0,R))
=O(R^\nu)\\
|\operatorname{Rm}_h|=O(R^{-q})\quad\text{on } M\setminus B_h(p_0,R)
\end{cases}\text{as }R\to\infty,\quad
\nu-2-\frac{2q}{3}&<0.
\end{aligned}
\end{equation}
Then, after at worst passing to a double cover of $M$, $h$ is conformal to a K\"ahler metric with positive scalar curvature. In particular, $h$ is Ricci-flat,
Hermitian and non-K\"ahler.
\end{theorem}
As far as gravitational instantons are concerned, we can obtain two rigidity results. The first is a direct consequence of Theorem \ref{thm:intro-extensions}: indeed, recalling the celebrated work of Bando, Kasue and Nakajima \cite{BandoKasueNakajima}, we show that $(M,h)$ is a Type II gravitational instanton (i.e. Hermitian non-K\"ahler) having an asymptotically locally Euclidean (ALE) end, under a non-collapsing hypothesis on the volume of geodesic balls (see Corollary \ref{cor:noncollapsed-gravitational-instanton}). Furthermore, we also apply the zero-capacity argument 
to prove a classification result for collapsed gravitational instantons, under the same hypothesis as in Theorem \ref{thm:main}, obtaining the examples found by Biquard and Gauduchon \cite{BiquardGauduchon}, recently classified as the only collapsed Type II instantons with an asymptotically locally flat (ALF) end by Li \cite{LiGravInstantons} (Theorem \ref{thm:collapsed-gravitational-instanton}). We highlight the fact that, as explained later (see Remark \ref{rem:cheegertian}), this result can be seen as a direct consequence of Theorem \ref{thm:intro-extensions} and a celebrated result due to Cheeger and Tian \cite{CheegerTian}: however, the proof we included here only relies on the zero-capacity argument and H\"older's inequality. Similar techniques involving the use of the weighted Weitzenb\"ock formula \eqref{eq:weighted-weitzenbock} (but without the zero-capacity argument) are also exploited in \cite{BiquardGauduchonLeBrun}. 

As for the compact case, the uniform bound $\beta_h\leq\rho_0\alpha_h$, with $\rho_0<1$, is optimal: indeed, we show that there exist \emph{multicentered Gibbons--Hawking gravitational instantons} $(M,h)$ \cite{GibbonsHawking, GibbonsHawking2} with tetrahedral symmetry which, if we consider the orientation opposite to the one induced by the hyperk\"ahler structure, admit points where $\alpha_h=\beta_h$ and for which the conclusion of Theorem \ref{thm:intro-extensions} does not hold. In this case, we also describe an explicit loss of continuity phenomenon for $\rho$ around a point where $W_h^+=0$. 

The paper is organized as follows. In Section \ref{sec:confchange}, we provide some preliminaries and 
describe LeBrun's conformal normalization: furthermore, we prove the weighted
Derdzi\'nski identities and the crucial first-order identities \eqref{eq:c-relations}. In Section \ref{sec:proof}, we prove Theorem
\ref{thm:half-harmonic-simple-top} (and, therefore, Theorem \ref{thm:Einstein-simple-top}) and Corollary
\ref{cor:det-version-hhw}. Section \ref{sec:closed-einstein} deals with the
zero set of $W_h^+$: here we prove Theorem \ref{thm:main} and its consequences (Corollary \ref{cor:det-version}, Corollary \ref{cor:nonpositive-scalar-nonexistence}). Section \ref{sec:ke-holomorphic-pinching}
is devoted to the proof of Theorem \ref{thm:siu-yang-pinching}. Section
\ref{sec:extensions} treats complete Ricci-flat manifolds and
gravitational instantons: here we prove Theorem \ref{thm:intro-extensions} and the classification results for noncollapsed (Corollary \ref{cor:noncollapsed-gravitational-instanton}) and collapsed (Theorem \ref{thm:collapsed-gravitational-instanton}) gravitational instantons. Finally, in Section
\ref{sec:ricci-flat-endpoint-model} we construct the compact and
noncompact examples for which the largest eigenvalue has multiplicity
two, which show the optimality of our results. 
\begin{remark}
    We point out that, given any oriented Riemannian four-manifold $(M,h)$, reversing the orientation causes the subbundles $\Lam^+$ and $\Lam^-$ to swap, which implies that 
    the operators $W_h^+$ and $W_h^-$ are exchanged. Hence, all the results stated for $W_h^+$ can also be written in terms of $W_h^-$: the only important caveat is that, if we choose to work with the anti-self-dual Weyl curvature $W_h^-$, the conformal metric $g=\alpha_h^{2/3}h$ will be K\"ahler with respect to the \emph{opposite} orientation (here $\alpha_h$ is the largest eigenvalue of $W_h^-$ with respect to the previous orientation).
\end{remark}

\section*{Acknowledgements}
The authors would like to thank Andrea Malchiodi and Benedetta Piroddi for interesting and helpful discussions regarding the examples described in Section \ref{sec:ricci-flat-endpoint-model}. Both authors are members of the Gruppo Nazionale per le Strutture Algebriche, Geometriche e le loro Applicazioni (GNSAGA) of the Istituto Nazionale di Alta Matematica (INdAM). 
	
	\section{LeBrun's conformal change} \label{sec:confchange}
	
	In this section, we adapt a strategy due to 
	LeBrun \cite{LeBrun2021} based on 
	a specific conformal change of the metric $h$. 
    First, we recall some well-known facts about four-dimensional Riemannian geometry (see e.g. \cite{Besse} for more detailed discussions). 
    Let $(M,h)$ be a connected, oriented Riemannian four-manifold: the action of the Hodge operator $*=*_h$ on the bundle of 2-forms $\Lam^2$ induces the conformally invariant splitting $\Lam^2=\Lam_h^+\oplus\Lam_h^-$, where, at every point, $\Lam_h^\pm$ is the eigenspace of $*$ associated to the eigenvalue $\pm 1$. We call $\Lam_h^\pm$ the subbundles of \emph{self-dual} and \emph{anti-self-dual} forms, respectively, and we say that a 2-form $\omega$ is self-dual (respectively, anti-self-dual) if
    $*\omega=\omega$ (respectively, if $*\omega=-\omega$): 
    by conformal invariance of the Hodge operator on 2-forms, we know that 
    the subbundles $\Lam_h^\pm$ are conformally invariant, hence, throughout the paper, we will write $\Lam^\pm=\Lam_{\overline{h}}^\pm$ for every $\overline{h}\in [h]$. 
    Since the Weyl tensor $W_h$ preserves the subbundles, this decomposition also induces a splitting of $W_h$ into a self-dual and an anti-self-dual part, denoted respectively by 
    $W_h^+$ and $W_h^-$. Both can be regarded as linear, trace-free, symmetric endomorphisms of the subbundles $\Lam^\pm$ and one has $*W_h^\pm=\pm W_h^\pm$: equivalently, $W_h^+$ and $W_h^-$ may be viewed as a self-dual and an anti-self-dual $\Lam^\pm$-valued 2-form, respectively.  
    
    We say that $(M,h)$ has \emph{harmonic Weyl curvature} if 
    $\delta_hW_h=0$ on $M$, where $\delta_h$ denotes the
    divergence operator: since $M$ is four-dimensional, the splitting of $W_h$ allows us to define metrics with 
    harmonic (anti-)self-dual Weyl curvature, i.e. metrics satisfying $\delta_hW_h^\pm=0$ on $M$. 
    The harmonic Weyl condition enjoys a weighted conformal covariance property (see also \cite{LeBrun2021, PenroseRindler2}):
    \begin{lemma}
		\label{divergenceconfchange}
		Let $(M,h)$ be a Riemannian four-manifold and let $g=f^{-2}h$, for some smooth positive function $f$ on $M$. 
		Then,
		\[
		\delta_g(fW_g)=f(\delta_hW_h).
		\]
        In particular, we have that
            \begin{equation}\label{eq:conformal-divergence}
			\delta_h W_h=0
			\quad\Longleftrightarrow\quad
			\delta_g(fW_g)=0;
		\end{equation}  
		furthermore, \eqref{eq:conformal-divergence} holds independently for $W_h^{\pm}$.
	\end{lemma}

    \begin{proof}
		First, let $g=e^{2u}h$, $u\in C^{\infty}(M)$: by the conformal invariance of the $(1,3)$ version of
		the Weyl tensor, we obtain (see also \cite{CatinoMastroliaBook})
		\[
		\delta_gW_g=\delta_hW_h
		+W_h(\nabla^hu, \cdot, \cdot, \cdot).
		\]
		An easy computation then shows that, for every
		smooth function $f$ on $M$,
		\begin{equation*}
		\delta_g(fW_g)=W_h(\nabla^hf, \cdot, \cdot, \cdot)
		+f[\delta_hW_h+W_h(\nabla^hu, \cdot, \cdot, \cdot)].
		\end{equation*}
		Now, putting $f=e^{-u}$, we immediately obtain
		\begin{align*}
			\delta_g(fW_g)=&W_h(\nabla^hf, \cdot, \cdot, \cdot)+f\left[\delta_hW_h-
			\dfrac{1}{f}W_h(\nabla^hf, \cdot, \cdot, \cdot)\right]\\
			=&f\delta_hW_h,
		\end{align*}
		which proves the claims. The same holds for $W_h^{\pm}$ since the 
		Hodge operator commutes with the covariant
		derivative induced by the Levi-Civita connection.
	\end{proof}
    \begin{remark}
        One can actually prove a more general result: if $(M,h)$ is a Riemannian manifold of dimension $n\geq4$, then 
        $\delta_g(f^{n-3}W_g)=f^{n-3}(\delta_hW_h)$, where
        $g=f^{-2}h$. 
    \end{remark}
    From now on, we focus on $W_h^+$: since, at every point 
    $p\in M$, $W_h^+$ is a symmetric, linear map from 
    $(\Lam^+)_p$ to itself, we can define its eigenvalues
    \[
    \alpha_h\geq\beta_h\geq\gamma_h.
    \]
	Let $U\subset M$ be an open set on which
	$\alpha_h>\beta_h$: then, $\alpha_h$ is a smooth positive function on $U$. 
    We set
	\[
	f=\alpha_h^{-1/3},
	\quad
	g=f^{-2}h=\alpha_h^{2/3}h.
	\]
	By the conformal invariance of $*$ and $W_h^+$, viewed as a 
    $(1,3)$-tensor, the eigenvalues of $W_g^+:\Lam^+\to\Lam^+$ are given by
	\[
	\alpha=f^2\alpha_h,
	\quad
	\beta=f^2\beta_h,
	\quad
	\gamma=f^2\gamma_h.
	\]
	Consequently,
	\[
	f\alpha=1,
	\quad
	\frac{\beta}{\alpha}=\frac{\beta_h}{\alpha_h}=\rho,
	\quad
	\gamma=-(1+\rho)\alpha.
	\]
	All the geometric quantities below are computed with respect to $g$, unless
	explicitly stated otherwise.
%	This is the weighted conformal normalization used by LeBrun in his
%	conformal proof of Wu's theorem, building on Derdzi\'nski's construction
%	\cite{Derdzinski,LeBrun2021}.  
	We assume that $h$ has harmonic self-dual Weyl curvature,
\[
    \delta_hW_h^+=0
\]
and we introduce
	\[
	\calW^+=fW^+,
	\]
	whose eigenvalues are
	\[
	\mu_1=1,
	\quad
	\mu_2=\rho,
	\quad
	\mu_3=-(1+\rho);
	\]
	in particular, the largest eigenvalue $\mu_1$ is constant. Moreover, by \eqref{eq:conformal-divergence},
\[
    \delta\mathcal W^+=0.
\]
Since $\alpha_h>\beta_h$ on $U$, the eigenspace of $W_h^+$ associated to $\alpha_h$ varies smoothly on $U$, hence determining a smooth, real line bundle 
	\[
	L\subset\Lam^+|_U
	\]
    over $U$. As explained in \cite[Section 2]{LeBrun2021}, $L$ need not to be orientable; however, 
    its sphere bundle
    \[
    \widehat U:=\{\omega\in L: \norm{\omega}^2=2\}
    \]
    is a double cover of $U$ and the pullback of $L$ to
    $\widehat U$ has a global smooth section $\omega_1$
    such that
    \[
	\norm{\omega_1}^2=2,
	\quad
	W^+(\omega_1)=\alpha\omega_1;
	\]
    Therefore, up to passing to this double cover, we can always assume that $L$ is orientable and, hence, trivial; 
    for notational simplicity, whenever we pass to such a cover we use the same symbols for the pulled-back tensors and metrics.
    The normalization $\norm{\omega_1}^2=2$ is chosen since, given any point $p$, 
    a self-dual form $\eta\in(\Lam^+)_p$ with $\norm{\eta}^2=2$ corresponds to an almost complex structure $J$, compatible with $g$ and the orientation, on $T_pM$ such that $\eta=g(J\cdot, \cdot)$.

	Following Derdzi\'nski \cite{Derdzinski}, let $M_W$ denote the regular spectral set of $W_g^+$,
namely the set of points at which the number of distinct eigenvalues
is locally constant. Let $M_W^U:=M_W\cap U$. Since $\alpha>\beta$ on $U$, if
\[
    Z:=\{x\in U:\beta(x)=\gamma(x)\},
\]
then
\[
    M_W^U
    =
    (U\setminus Z)\cup\operatorname{int}_U Z
    =
    U\setminus\partial_U Z.
\]
Observe that $Z$ is closed in $U$ and its boundary has empty interior; thus
$M_W^U$ is open and dense in $U$. 
	Locally, we can complete the section $\omega_1$ to an oriented orthogonal
	frame $\{\omega_1,\omega_2,\omega_3\}$ of $\Lam^+$, which we can express as
	\begin{equation} \label{eq:localeigenframe}
	\omega_1=e^{12}+e^{34},
	\quad
	\omega_2=e^{13}-e^{24},
	\quad
	\omega_3=e^{14}+e^{23},
	\end{equation}
	where $e^{ij}=e^i\wedge e^j$ and $\{e^i\}$ is a local oriented $g$-orthonormal coframe.
	Near any point at which $\beta\neq\gamma$, the forms $\omega_2$ and $\omega_3$ are chosen to be
	smooth eigenforms of $W^+$, i.e.,
	\[
	W^+(\omega_2)=\beta\omega_2,
	\quad
	W^+(\omega_3)=\gamma\omega_3;
	\]
	on the other hand, on $\operatorname{int}_{U}Z$ the restriction of $W^+$ to
	$L^\perp$ is a multiple of the identity, so any smooth oriented orthogonal
	frame of $L^\perp$ is an eigenframe. This observation allows us to
	perform all the following local computations on $M_W^U$. Since $\norm{\omega_1}$ is constant, there are locally defined one-forms $a$ and $c$ \cite{Derdzinski} such that 
	\begin{equation*}
		\nabla\omega_1=a\otimes\omega_2-c\otimes\omega_3,
	\end{equation*}
	where
	\[
	a=\sum_{i=1}^4 a_i e^i,
	\quad
	c=\sum_{i=1}^4 c_i e^i. 
	\]
	We note that, on $M_W^U$, the completion of $\omega_1$
	to an orthogonal basis defined by 
	\eqref{eq:localeigenframe} also provides the equations
	for the covariant derivatives of $\omega_2$ and
	$\omega_3$, which yield the system
	\begin{equation} \label{eq:conneigenframe}
    \begin{cases}
		\nabla\omega_1=a\otimes\omega_2-c\otimes\omega_3, \\
		\nabla\omega_2=-a\otimes\omega_1+b\otimes\omega_3\\
		\nabla\omega_3=c\otimes\omega_1-b\otimes\omega_2,
        \end{cases}
	\end{equation}
	where $b$ is a one-form locally defined by
	$b=\sum_{i=1}^4b_ie^i$ (see \cite[Equation (32)]{Derdzinski}).

    We now prove new fundamental first-order identities on the
    coefficients of $\omega_i$, in the spirit of Derdzi\'nski's local computations \cite{Derdzinski}.
	\begin{proposition}\label{prop:weighted-derdzinski}
		On $M_W^U$, one has
		\begin{equation}\label{eq:weighted-relation}
			(1-\rho)\iota_{a^\sharp}\omega_3+(2+\rho)\iota_{c^\sharp}\omega_2=0.
		\end{equation}
		Equivalently, if
		\[
		k=\frac{1-\rho}{2+\rho},
		\]
		then
		\begin{equation}\label{eq:c-relations}
			c_2=ka_1,
			\quad
			c_1=-ka_2,
			\quad
			c_4=ka_3,
			\quad
			c_3=-ka_4.
		\end{equation}
	\end{proposition}
	
	\begin{proof}
		Fix $p\in M_W^U$ and a neighborhood $V$ on which the eigenframe
$\{\omega_1,\omega_2,\omega_3\}$ is smooth. Let 
\[
    A=\frac12\sum_{r=1}^3\mu_r\,\omega_r\otimes\omega_r
\]
be a divergence-free section of $S_0^2\Lam^+$ which is diagonal in this frame, i.e.
		\[
		A(\omega_r)=\mu_r\omega_r,
		\quad
		\norm{\omega_r}^2=2.
		\]
		Equivalently,
		\[
		A_{ijkl}=\frac12\sum_{r=1}^3\mu_r(\omega_r)_{ij}(\omega_r)_{kl}.
		\]
		We compute at a fixed point and choose the tangent frame $e_1,\ldots,e_4$ to be geodesic there: we use
		the fact that, for every $p\in M_W^U$, the operator $W^+$ can be smoothly diagonalized in an open neighborhood of $p$. The condition $\delta A=0$ is
		\[
		0=(\delta A)_{jkl}=-\nabla_iA_{ijkl},
		\]
		i.e.
		\begin{equation} \label{vanishdiverA}
			0=\sum_{i=1}^4\sum_{r=1}^3
			\left[
			d\mu_r(e_i)(\omega_r)_{ij}(\omega_r)_{kl}
            +\mu_r(\nabla_i\omega_r)_{ij}(\omega_r)_{kl}
			+\mu_r(\omega_r)_{ij}(\nabla_i\omega_r)_{kl}
			\right].
		\end{equation}
		First, we compute \eqref{vanishdiverA} putting
		$j=l=2$, $k=1$, which yields 
		\begin{align*}
			0&=\sum_{i=1}^4\sum_{r=1}^3
			\left[
			d\mu_r(e_i)(\omega_r)_{i2}(\omega_r)_{12}
			+\mu_r(\nabla_i\omega_r)_{i2}(\omega_r)_{12}
			+\mu_r(\omega_r)_{i2}(\nabla_i\omega_r)_{12}
			\right]\\
			&=:I+II+III.
		\end{align*}
		By \eqref{eq:localeigenframe} and \eqref{eq:conneigenframe}, a direct computation gives
		\begin{align*}
			I&=\sum_{i,r}d\mu_r(e_i)(\omega_r)_{i2}(\omega_r)_{12}=
			d\mu_1(e_1);\\
			II&=\sum_{i=1}^4\mu_1(\nabla_i\omega_1)_{i2}=
			\sum_{i=1}^4\mu_1(a_i\omega_2-c_i\omega_3)_{i2}=
			\mu_1(a_4+c_3);\\
			III&=\mu_1(\nabla_1\omega_1)_{12}+
			\mu_2(\nabla_4\omega_2)_{12}-\mu_3(\nabla_3\omega_3)_{12}\\
			&=\mu_2(-a_4\omega_1+b_4\omega_3)_{12}-\mu_3(c_3\omega_1-b_3\omega_2)_{12}=-\mu_2a_4-\mu_3c_3;
		\end{align*}
		this implies that
		\[
		d\mu_1(e_1)=-(\mu_1-\mu_2)a_4-(\mu_1-\mu_3)c_3.
		\]
		Performing analogous computations with
		$(j,k,l)=(1,1,2), (3,1,2), (4,1,2)$, we also obtain 
		\begin{align*}
			d\mu_1(e_2)&=-(\mu_1-\mu_2)a_3+(\mu_1-\mu_3)c_4\\
			d\mu_1(e_3)&=(\mu_1-\mu_2)a_2+(\mu_1-\mu_3)c_1\\
			d\mu_1(e_4)&=(\mu_1-\mu_2)a_1-(\mu_1-\mu_3)c_2,
		\end{align*}
		i.e. 
		\begin{align*}
			d\mu_1=&[-(\mu_1-\mu_2)a_4-(\mu_1-\mu_3)c_3]e^1+
			[-(\mu_1-\mu_2)a_3+(\mu_1-\mu_3)c_4]e^2\\
			&+[(\mu_1-\mu_2)a_2+(\mu_1-\mu_3)c_1]e^3+
			[(\mu_1-\mu_2)a_1-(\mu_1-\mu_3)c_2]e^4\\
			=&(\mu_1-\mu_2)(-a_4e^1-a_3e^2+a_2e^3+a_1e^4)\\&+
			(\mu_1-\mu_3)(-c_3e^1+c_4e^2+c_1e^3-c_2e^4).
		\end{align*}
		Now, one has that
		\begin{align} \label{intproducts}
			\iota_{a^\sharp}\omega_3&=-a_4e^1-a_3e^2+a_2e^3+a_1e^4\\
			\iota_{c^\sharp}\omega_2&=-c_3e^1+c_4e^2+c_1e^3-c_2e^4, \notag
		\end{align}
		which means that 
		\begin{equation}\label{eq:derdzinski-first-order}
			d\mu_1=(\mu_1-\mu_2)\iota_{a^\sharp}\omega_3+(\mu_1-\mu_3)\iota_{c^\sharp}\omega_2;
		\end{equation}
        we point out that this equation also appears in 
        \cite[Section 2]{Wu} and it is a rewriting of
		Derdzi\'nski's first-order identity for the first eigenvalue \cite{Derdzinski}.
		
		We now apply \eqref{eq:derdzinski-first-order} to $A=\calW^+=fW^+$. By \eqref{eq:conformal-divergence}, $A$ is divergence-free, and its eigenvalues are
		\[
		\mu_1=1,
		\quad
		\mu_2=\rho,
		\quad
		\mu_3=-(1+\rho).
		\]
		Thus $d\mu_1=0$, and \eqref{eq:derdzinski-first-order} becomes
		\[
		0=(1-\rho)\iota_{a^\sharp}\omega_3+(2+\rho)\iota_{c^\sharp}\omega_2,
		\]
		which proves \eqref{eq:weighted-relation}. 
		Substituting 
		\eqref{intproducts} into \eqref{eq:weighted-relation} and comparing coefficients yields \eqref{eq:c-relations}.
	\end{proof}
	
	\begin{remark}\label{rem:weighted}
The weight $f$ is crucial for removing the terms involving $d\alpha$
from \eqref{eq:derdzinski-first-order}. Indeed,
$f=\alpha_h^{-1/3}$ is precisely the conformal factor for which
$\delta(fW^+)=0$ and $\mu_1=1$.
	\end{remark}
	
	We shall also use the Weitzenb\"ock formula for divergence-free sections of
	$S^2_0\Lam^+$ in the weighted conformal form exploited by LeBrun
	\cite[Equations (8)--(9)]{LeBrun2021}, derived from the standard
	identity for metrics with harmonic self-dual Weyl curvature due to Derdzi\'nski \cite{Derdzinski}. Since $\mathcal{W}^+=fW^+$ is divergence-free, we have 
    \[
    \delta^\nabla \mathcal{W}^+\equiv 0,
    \]
    where we view $\mathcal{W}^+$ as a $\Lam^+$-valued 2-form, 
    $\delta^\nabla=-*d^\nabla *$ and $d^\nabla$ is the exterior covariant derivative induced by the Levi-Civita connection on 
    $\Lam^+$: since $\mathcal{W}^+$ is self-dual, we also obtain 
    \[
    0=*(\delta^\nabla \mathcal{W}^+)=-*(*d^\nabla *(\mathcal{W}^+))=d^\nabla \mathcal{W}^+,
    \]
    which implies that $\Delta_H\mathcal{W}^+\equiv 0$, where 
    $\Delta_H=d^\nabla\delta^\nabla+\delta^\nabla d^\nabla$ is
    the Hodge Laplacian. Thus, by exploiting the classical
    Weitzenb\"{o}ck formula for trace-free, symmetric 
    bundle endomorphisms (see e.g. \cite[Proposition 4.2]{Bourguignon}), we obtain the well-known
	
	\begin{lemma}\label{lem:weighted-weitzenbock}
		Let $(M,g)$ be an oriented Riemannian four-manifold, and let $f$ be a
		positive smooth function such that
		\[
		\delta_g(fW_g^+)= 0\quad\text{on }M.
		\]
		Then
		\begin{equation}\label{eq:weighted-weitzenbock}
			0=\nabla^*\nabla(fW_g^+)+\frac{s_g}{2}fW_g^+
			-6fW_g^+\circ W_g^+ +2f\norm{W_g^+}^2I,
		\end{equation}
		where $I$ denotes the identity endomorphism of $\Lam^+$.
	\end{lemma}
    \section{Proof of Theorem \ref{thm:half-harmonic-simple-top} and Corollary \ref{cor:det-version-hhw}}\label{sec:proof}

    \begin{proof}[Proof of Theorem \ref{thm:half-harmonic-simple-top}] By the discussion at the beginning
    of Section \ref{sec:confchange}, we know that the 
    subbundle $L\subset\Lam^+$ defined by the eigenspace of $\alpha_h$ 
    is smooth and, after passing, if necessary, to a double cover of $M$,
    we may assume that $L$ is trivial. As before, we set 
    \[
    g=\alpha_h^{2/3}h.
    \]
   From now on, unless specified otherwise, all geometric quantities are
computed with respect to $g$; hence we write
\[
\Lam^+=\Lam_g^+, \quad \alpha=\alpha_g, \quad W^+=W_g^+, \quad s=s_g.
\]
    Furthermore, we recall that, for every $\eta_1,\eta_2\in\Lam^+$, 
    \[
    \ip{W^+(\eta_1)}{\eta_2}=W^+(\eta_1,\eta_2)=\ip{W^+}{\eta_1\otimes\eta_2}.
    \]
    Choose a globally defined smooth section $\omega_1$ of $L$ such that
    \[
    W^+(\omega_1)=\alpha\omega_1, \quad
    \norm{\omega_1}^2=2;
    \]
    our first step is to derive a pointwise identity, also appearing implicitly in \cite{LeBrun2021} (see also \cite{LeBrun2015}),
    using the weighted Weitzenb\"{o}ck formula \eqref{eq:weighted-weitzenbock}.
    \begin{lemma}
		At every point of $M$, we have
\begin{equation}\label{eq:pointwise-before-hodge}
\begin{aligned}
0={}&\frac12\norm{\nabla\omega_1}^2
-\frac{2}{\alpha}W^+(\nabla_e\omega_1,\nabla^e\omega_1)\\
&+\frac32\ip{\omega_1}{(d+d^*)^2\omega_1}
+2\left(\frac{2\norm{W^+}^2}{\alpha}-3\alpha\right).
\end{aligned}
\end{equation}
 	\end{lemma}
    \begin{proof}
		We claim that, at every $p\in M$, 
		\begin{equation} \label{eq:prelimweitzen}
					\ip{\nabla^*\nabla(fW^+)}{\omega_1\otimes\omega_1}
					=
					2\ip{\omega_1}{\nabla^*\nabla\omega_1}
					-\frac{2}{\alpha}W^+(\nabla_e\omega_1,\nabla^e\omega_1).
				\end{equation}
		We observe that, since $\norm{\omega_1}$ is constant and $fW^+(\omega_1)=\omega_1$, 
		\[
		\nabla(fW^+(\omega_1,\omega_1))=
		\nabla\left(\ip{fW^+(\omega_1)}{\omega_1}\right)=
		\nabla\ip{f\alpha\omega_1}{\omega_1}=
		\nabla\norm{\omega_1}^2=0;
		\]
		now, recalling that 
		$\nabla^*\nabla=-\nabla_e\nabla^e$, we use the standard formula for covariant
		derivatives, together with the fact that
		$fW^+$ and $\nabla(fW^+)$ are self-adjoint, 
		to obtain
		\begin{align} \label{eq:first-weitzen}
		\ip{\nabla^*\nabla(fW^+)}{\omega_1\otimes\omega_1}&=
		4\ip{\nabla_e (fW^+)(\omega_1)}{\nabla^e\omega_1}\\
		&+2\ip{fW^+(\nabla_e\omega_1)}{\nabla^e\omega_1}
		-2\ip{\omega_1}{\nabla^*\nabla\omega_1}. \notag
		\end{align}
		We use again the fact that $fW^+(\omega_1)=\omega_1$: if
		we differentiate this equation, we obtain
		\[
		\nabla_e\omega_1=\nabla_e(fW^+(\omega_1))=
		\nabla_e(fW^+)(\omega_1)+fW^+(\nabla_e\omega_1)
		\]
		and, if we contract the previous identity with 
		$\nabla^e\omega_1$, we get
		\[
		\ip{\nabla_e (fW^+)(\omega_1)}{\nabla^e\omega_1}=
		\norm{\nabla\omega_1}^2-\ip{fW^+(\nabla_e\omega_1)}{\nabla^e\omega_1}.
		\]
		Furthermore, by the usual Bochner formula for 
		smooth tensors,
		\[
		0=\nabla^*\nabla\norm{\omega_1}^2=
		2\ip{\omega_1}{\nabla^*\nabla\omega_1}
		-2\norm{\nabla\omega_1}^2 \quad
		\Longrightarrow \quad
		\norm{\nabla\omega_1}^2=\ip{\omega_1}{\nabla^*\nabla\omega_1}.
		\]
	Substituting these identities into \eqref{eq:first-weitzen} and using
$f\alpha=1$ proves the claim. Contracting
\eqref{eq:weighted-weitzenbock} with $\omega_1\otimes\omega_1$ and
using again $f\alpha=1$, $\norm{\omega_1}^2=2$, and
\eqref{eq:prelimweitzen}, we obtain
		\begin{align*}
			0&=2\ip{\omega_1}{\nabla^*\nabla\omega_1}
			-\dfrac{2}{\alpha}W^+(\nabla_e\omega_1,\nabla^e\omega_1)+s-12\alpha+\dfrac{4}{\alpha}\norm{W^+}^2\\
			&=\dfrac{1}{2}\norm{\nabla\omega_1}^2+
			\dfrac{3}{2}\ip{\omega_1}{\nabla^*\nabla\omega_1}
			-\dfrac{2}{\alpha}W^+(\nabla_e\omega_1,\nabla^e\omega_1)+s-12\alpha+\dfrac{4}{\alpha}\norm{W^+}^2.
		\end{align*}
		Finally, by the well-known Weitzenb\"ock formula for self-dual two-forms (see e.g. \cite[Equation (7)]{LeBrun2021})
		\[
		(d+d^*)^2\omega_1
		=
		\nabla^*\nabla\omega_1-2W^+(\omega_1)+\frac{s}{3}\omega_1,
		\]
		substitution gives \eqref{eq:pointwise-before-hodge}.
	\end{proof} 
    We now integrate \eqref{eq:pointwise-before-hodge} over $M$:
    first, we
    point out that, since $*\omega_1=\omega_1$ and $d^{*}\omega_1=-*d*\omega_1$,
	using the properties of the Hodge operator we obtain
	\begin{equation} \label{eq:Hodge-lapl-self-dual}
		\int_M\ip{\omega_1}{(d+d^*)^2\omega_1}\dd\mu_g
		=2\int_M\norm{d\omega_1}^2\dd\mu_g.
	\end{equation}
    Then, setting again 
    \[
    \rho=\dfrac{\beta}{\alpha}=\dfrac{\beta_h}{\alpha_h}
    \]
    and using \eqref{eq:pointwise-before-hodge} and \eqref{eq:Hodge-lapl-self-dual}, together with 
    the fact that $\norm{W^+}^2=2\alpha^2(1+\rho+\rho^2)$, we get the equality
    \begin{equation}\label{eq:domain-identity}
		\int_M
		\left[
		\mathcal P
		+
		8\alpha\left(\rho+\frac12\right)^2
		\right]\dd\mu_g=0,
	\end{equation}
    where
	\begin{equation} \label{eq:P-identity-1}
	\mathcal P
	:=
	3\norm{d\omega_1}^2
	+\frac12\norm{\nabla\omega_1}^2
	-\frac{2}{\alpha}W^+(\nabla_e\omega_1,\nabla^e\omega_1).
	\end{equation}
 The crucial new point is that the weighted Derdzi\'nski relation in
Proposition \ref{prop:weighted-derdzinski} allows us to evaluate the
indefinite gradient term $\mathcal P$ exactly, rather than estimate it.
Let $M_W\subset M$ be the regular spectral set for $W^+$, equivalently
for $\calW^+=fW^+$. Since $\alpha>\beta$ and $\alpha>0$ by hypothesis,
the discussion in Section \ref{sec:confchange} implies that $M_W$ is an open dense subset of $M$ and
    that Proposition \ref{prop:weighted-derdzinski} holds
    locally on the whole $M_W$.
    
    Indeed, we can complete $\omega_1$ to the
    oriented orthogonal frame $\{\omega_1, \omega_2, \omega_3\}$ of $\Lam^+$ defined in 
    \eqref{eq:localeigenframe} locally on $M_W$: the covariant derivatives
    of the forms $\omega_i$ are then given by 
    \eqref{eq:conneigenframe}, which 
    immediately implies that
    \[
    \norm{\nabla\omega_1}^2=2\left(\norm{a}^2+\norm{c}^2\right)
    \]
    on $M_W$. 
    A direct computation gives 
    \begin{align*}
    d\omega_1&=a\wedge\omega_2-c\wedge\omega_3\\
    &=-(a_2+c_1)e^{123}+(-a_1+c_2)e^{124}+(a_4+c_3)e^{134}+
    (a_3-c_4)e^{234},
    \end{align*}
    where $e^{ijk}=e^i\wedge e^j\wedge e^k$, which implies that
    \[
    \norm{d\omega_1}^2=(a_1-c_2)^2+(a_2+c_1)^2+(a_3-c_4)^2+
    (a_4+c_3)^2.
    \]
    Finally, using the first equation in \eqref{eq:conneigenframe} again, we can compute
    \[
    W^+(\nabla_e\omega_1,\nabla^e\omega_1)=
    2\left(\beta\norm{a}^2+\gamma\norm{c}^2\right);
    \]
    by
    \eqref{eq:c-relations}, locally on $M_W$ we have
        \begin{equation} \label{eq:local-identities-omega1}
        \begin{cases}
        \norm{d\omega_1}^2=(1-k)^2\norm{a}^2,\\
    \norm{\nabla\omega_1}^2=2(1+k^2)\norm{a}^2,\\
    W^+(\nabla_e\omega_1,\nabla^e\omega_1)=2\alpha\left(\rho-(1+\rho)k^2\right)\norm{a}^2,
    \end{cases}
    \end{equation}
    where 
    \[
    k=\dfrac{1-\rho}{2+\rho}.
    \]
    Now, we can express $\mathcal{P}$ using \eqref{eq:local-identities-omega1} as
    \begin{align*}
        \mathcal{P}&=[3(1-k)^2+(1+k^2)-4(\rho-(1+\rho)k^2)]\norm{a}^2\\
        &=[4(1-\rho)-6k+4(2+\rho)k^2]\norm{a}^2\\
        &=\dfrac{6(1-\rho)}{2+\rho}\norm{a}^2.
    \end{align*}
    Moreover, \eqref{eq:local-identities-omega1} gives
    \[
    \norm{\nabla\omega_1}^2=2(1+k^2)\norm{a}^2=
    \dfrac{2(2\rho^2+2\rho+5)}{(2+\rho)^2}\norm{a}^2,
    \]
 and hence
\begin{equation}\label{eq:P-identity-2}
\mathcal{P}=\mathcal{Q}_\rho\norm{\nabla\omega_1}^2,
\quad
\mathcal{Q}_\rho:=\dfrac{3(1-\rho)(2+\rho)}{2\rho^2+2\rho+5}.
\end{equation}
This identity holds on $M_W$.

By \eqref{eq:P-identity-1}, $\mathcal P$ is smooth on $M$, since
$\omega_1$ is a globally defined smooth eigenform. Moreover, $\rho$ is
continuous on $M$, so
$\mathcal Q_\rho\norm{\nabla\omega_1}^2$ is globally defined and
continuous. These two functions coincide on the open dense set $M_W$;
therefore,
\[
\mathcal{P}=\mathcal{Q}_\rho\norm{\nabla\omega_1}^2
\quad\text{on }M.
\]
Consequently, \eqref{eq:domain-identity} becomes
    \begin{equation} \label{eq:fundamental-identity}
    \int_M\left[\dfrac{3(1-\rho)(2+\rho)}{2\rho^2+2\rho+5}\norm{\nabla\omega_1}^2+8\alpha\left(\rho+\dfrac{1}{2}\right)^2\right]\dd\mu_g=0.
    \end{equation}
 Since $-\frac12\leq\rho<1$, both terms in the integrand are
nonnegative and hence vanish identically. Since $\alpha>0$ and
$\rho<1$ on $M$, we conclude that
    \[
    \nabla\omega_1= 0,\quad
    \rho=-\dfrac{1}{2} \quad\text{on }M,
    \]
    which means that $g$ is a K\"{a}hler metric on $M$
    \cite[Proposition 5]{Derdzinski}, whose K\"{a}hler form is $\omega_1$ and
    whose self-dual Weyl operator $W^+$ has eigenvalues given by
    (see \cite[Proposition 2]{Derdzinski})
    \begin{equation} \label{eq:eigenvalues-Kahler}
    \alpha=\dfrac{s}{6}, \quad \beta=\gamma= -\dfrac{s}{12}.
    \end{equation}
    Furthermore, since $\alpha>0$ on $M$, by \eqref{eq:eigenvalues-Kahler} we obtain that
    $g$ is a K\"{a}hler metric with positive scalar curvature
    on $M$.
    
Conversely, let $(M,g,J)$ be a K\"ahler surface with $s_g>0$ and set
$h=s_g^{-2}g$. By \cite[Proposition 2]{Derdzinski},
$s_g^{-1}W_g^+$ is parallel, and hence
$\delta_g(s_g^{-1}W_g^+)=0$. Then, 
\eqref{eq:conformal-divergence} implies
$\delta_hW_h^+=0$, while \eqref{eq:eigenvalues-Kahler} gives
$\alpha_h>0$ and $\alpha_h>\beta_h$. This concludes the proof of
Theorem \ref{thm:half-harmonic-simple-top}.
    \end{proof}

    \begin{proof}[Proof of Corollary \ref{cor:det-version-hhw}]
		The strict inequality rules out zeros of $W^+_h$. Hence $W^+_h$ is nowhere zero and $\alpha_h>0$, $-1/2\leq\rho\leq 1$, and
		$\gamma_h=-(1+\rho)\alpha_h$. Hence, we immediately obtain
		\[
		|W_h^+|^2=2\alpha_h^2(1+\rho+\rho^2),
		\quad
		\det(W_h^+)=-\rho(1+\rho)\alpha_h^3,
		\]
		which imply that
        \[
		\frac{\det(W_h^+)}{|W_h^+|^3}=\mathcal{F}(\rho),
		\]
        where
		\[
		\mathcal{F}(\rho):=
		-\frac{\rho(1+\rho)}{2^{3/2}(1+\rho+\rho^2)^{3/2}}.
		\]
A direct computation gives
\[
\mathcal{F}'(\rho)=
\frac{(\rho-1)(\rho+2)\left(\rho+\frac12\right)}
{2^{3/2}(1+\rho+\rho^2)^{5/2}}
\leq0.
\]
Thus, $\mathcal{F}$ is non-increasing on $[-1/2,1]$, and its minimum
is attained at $\rho=1$, where
\[
\mathcal{F}(1)=-\frac1{3\sqrt6}.
\]
Since $\mathcal{F}(\rho)>\mathcal{F}(1)$ by assumption, we have
$\rho<1$ on $M$, and the conclusion follows from Theorem
\ref{thm:half-harmonic-simple-top}.   
    \end{proof}    

    \section{Compact Einstein manifolds} \label{sec:closed-einstein}
    
	The main aim of this section is to show 
    how the method used to prove Theorem \ref{thm:half-harmonic-simple-top} can be adapted to Einstein four-manifolds, without the non-vanishing assumption on 
    $W_h^+$. We prove that, if $W_h^+$ is not identically zero and its eigenvalues satisfy  
    \begin{equation}\label{eq:global-gap}
		\beta_h\leq\rho_0\alpha_h
		\quad\text{on }M,
		\quad
		\rho_0<1,
	\end{equation}
    then $W_h^+\neq 0$ at every point of $M$. The idea is to work on the set 
    \[
	X:=M\setminus N,
	\quad
	N:=\{W_h^+=0\}
	\]
    and, by constructing suitable cutoff functions satisfying a
zero-capacity condition, to prove that $g=\alpha_h^{2/3}h$ is K\"ahler
on $X$. We then use Derdzi\'nski's rigidity result for the zero set of
$W_h^+$ on Einstein four-manifolds
\cite[Proposition 5(iv)]{Derdzinski}.

    We first work under the weaker assumption $\delta_h W_h^+= 0$: 
    by \eqref{eq:global-gap}, on $X$ we have that $\alpha_h>\beta_h$ and, therefore, $\alpha_h>0$, which means that, using again Proposition \ref{prop:weighted-derdzinski}, 
    we can do the same algebraic computations we performed at the end of the previous section in order to obtain that \eqref{eq:pointwise-before-hodge} and 
    \eqref{eq:P-identity-2} hold on $X\cap M_W$ and, therefore, 
    on $X$ (recall that $X\cap M_W$ is an open dense subset of $X$). Since $X$ is not a closed set, we have to modify 
    our argument in order to use integration and obtain an analogue of 
    \eqref{eq:fundamental-identity}. 
   We begin with a general result for manifolds with harmonic self-dual
Weyl curvature.

    \begin{proposition}\label{prop:capacity-form}
		Let $(M,h)$ be an oriented four-manifold, not necessarily
		complete, such that $\delta_h W_h^+= 0$ and let
		\[
		N=\{W_h^+=0\},
		\quad
		X=M\setminus N .
		\]
		Assume that $W_h^+\not\equiv0$ and that, on $X$,
		\[
		\beta_h\leq\rho_0\alpha_h,
		\quad
		\rho_0<1 .
		\]
		  Let $g=\alpha_h^{2/3}h$. Suppose that there exist cutoff functions
		$\chi_j\in \operatorname{Lip}_c(X)$, $0\leq\chi_j\leq1$, such that
		$\chi_j\to1$ locally uniformly on $X$ and
		\begin{equation}\label{eq:abstract-capacity-condition}
			\int_X |d\chi_j|_g^2\,\dd\mu_g\longrightarrow0 .
		\end{equation}
		Then, after at worst passing to a double cover of $X$, 
       there exists a smooth eigenform
$\omega_1$, normalized by $|\omega_1|_g^2=2$, such that
\[
\nabla^g\omega_1=0,
\quad
\beta_h=\gamma_h=-\frac12\alpha_h
\quad\text{on }X
\]
and $g$ is K\"ahler with positive scalar curvature on $X$.
	\end{proposition}
	
	\begin{proof}
    Since $\delta_hW_h^+=0$,
    by \eqref{eq:conformal-divergence} we have
    $\delta_g(\alpha_h^{-1/3}W_g^+)=0$: hence, since $\rho$ is continuous on $X$, \eqref{eq:pointwise-before-hodge} and \eqref{eq:P-identity-2} hold on $X$. 
    Let $L\subset\restr{\Lam^+}{X}$ be the line subbundle associated with the eigenspace of $\alpha_h$. If $L$ is not orientable, we pass to its standard double cover; the capacity condition is unchanged up to the degree of the cover. Hence we may assume that $L$ is orientable and choose
    a smooth eigenform $\omega_1$ associated to $\alpha_h$ normalized by $\norm{\omega_1}_g^2=2$. If the line subbundle $L$ associated to the eigenspace of $\alpha_h$ is not orientable, 
 We recall that, since $\omega_1$ is self-dual, 
        \[
        d^{*}(\chi_j^2\omega_1)=-*d(\chi_j^2\omega_1)=
        \chi_j^2d^*\omega_1-2\chi_j*(d\chi_j\wedge\omega_1)
        \]
        almost everywhere on $X$.
        Hence, since $\norm{d^*\omega_1}=\norm{d\omega_1}$
        and $\norm{\omega_1}^2=2$, multiplying 
        \eqref{eq:pointwise-before-hodge} by $\chi_j^2$ and
        integrating by parts we obtain
		\begin{align} \label{eq:cutoff-hodge-error}
			\int_X\chi_j^2
			\left\langle\omega_1,(d+d^*)^2\omega_1\right\rangle\,\dd\mu_g =& 
			2\int_X\chi_j^2|d\omega_1|^2\,\dd\mu_g
			\\&+
			4\int_X\chi_j
			\left\langle d\chi_j\wedge\omega_1,d\omega_1\right\rangle
			\,\dd\mu_g.\notag			
		\end{align}
 Using again $\norm{\omega_1}^2=2$ and
$\norm{d\omega_1}\leq C\norm{\nabla\omega_1}$, the Cauchy--Schwarz
and Young inequalities give
\begin{align}
\left|\int_X\chi_j
\left\langle d\chi_j\wedge\omega_1,d\omega_1\right\rangle
\,\dd\mu_g\right|&\leq            C\int_X\norm{\chi_j}\norm{d\chi_j}\norm{\nabla\omega_1}\dd\mu_g\notag\\ 
            &\leq \varepsilon
            \int_X\chi_j^2\norm{\nabla\omega_1}^2\dd\mu_g+
            C_\varepsilon\int_X\norm{d\chi_j}^2\dd\mu_g
            \label{eq:cutoff-cross-term-estimate}
   \end{align}
   for every $\varepsilon>0$ and some $C_{\varepsilon}>0$.
        
		Combining \eqref{eq:pointwise-before-hodge}, \eqref{eq:P-identity-2}, and
		\eqref{eq:cutoff-hodge-error}, we obtain
		\begin{equation} \label{eq:cutoff-final-before-limit}
			0
			=
			\int_X\chi_j^2
			\left[
			\mathcal{Q}_{\rho}|\nabla\omega_1|^2
			+
			8\alpha\left(\rho+\frac12\right)^2
			\right]\dd\mu_g+
			6\int_X\chi_j
			\left\langle d\chi_j\wedge\omega_1,d\omega_1\right\rangle
			\dd\mu_g,
		\end{equation}
		where $\mathcal{Q}_\rho$ is defined in \eqref{eq:P-identity-2}: observe that, by the pinching assumption, 
        since $\mathcal{Q}_\rho$ is continuous and positive on
        $[-1/2,\rho_0]$, there exists a constant $q_0>0$, depending only on $\rho_0$, such that 
        \begin{equation}\label{eq:A-lower-bound-domain}
		\mathcal{Q}_\rho\geq q_0>0
		\quad\text{on }X.
	\end{equation}
        Set
		\[
		I_j:=
		\int_X\chi_j^2
		\left[
		\mathcal{Q}_\rho|\nabla\omega_1|^2
		+
		8\alpha\left(\rho+\frac12\right)^2
		\right]\dd\mu_g.
		\]
		  Using \eqref{eq:cutoff-cross-term-estimate}, \eqref{eq:cutoff-final-before-limit} and \eqref{eq:A-lower-bound-domain}, we get
		\[
		I_j
		\leq
		6\varepsilon\int_X\chi_j^2|\nabla\omega_1|^2\,\dd\mu_g
		+
		6C_\varepsilon\int_X|d\chi_j|^2\,\dd\mu_g
		\leq
		\frac{6\varepsilon}{q_0}I_j
		+
		6C_\varepsilon\int_X|d\chi_j|^2\,\dd\mu_g.
		\]
		Now, taking $\varepsilon>0$ small enough such that $6\varepsilon/q_0<1$, we obtain
		\[
		0\leq I_j\leq C\int_X|d\chi_j|^2\,\dd\mu_g,
		\]
        for some $C>0$. 
		By \eqref{eq:abstract-capacity-condition}, $I_j\to0$ and, since
		$\chi_j\to1$ locally uniformly on $X$ and the integrand of $I_j$ is nonnegative,
		Fatou's lemma yields
		\begin{equation}\label{eq:final-integral}
			0=
			\int_X
			\left[
			\mathcal{Q}_\rho|\nabla\omega_1|^2
			+
			8\alpha\left(\rho+\frac12\right)^2
			\right]\dd\mu_g .
		\end{equation}
	As in \eqref{eq:fundamental-identity}, both terms in the integrand
vanish identically. Hence
\[
\nabla^g\omega_1=0,
\quad
\beta_h=\gamma_h=-\frac12\alpha_h
\quad\text{on }X.
\]
Thus $g$ is K\"ahler on $X$ with $s_g>0$, which concludes the proof. 
	\end{proof}
Now, we are ready to prove our main result on compact Einstein
four-manifolds. 
\begin{proof}[Proof of Theorem \ref{thm:main}]
	We need to construct cutoff functions that satisfy the capacity hypothesis in
	Proposition \ref{prop:capacity-form} for the zero set of $W_h^+$. 
    
    If $N=\varnothing$, Proposition \ref{prop:capacity-form} applies with
$\chi_j\equiv1$ and shows that, after possibly passing to a double cover of $M$, $g=\alpha_h^{2/3}h$ is K\"ahler with positive
scalar curvature.

Henceforth suppose that $N\neq\varnothing$. The construction of the
cutoff functions relies on the fact that, under our hypotheses, the
zero set of $W_h^+$ cannot be too large.
    Since $(M,h)$ is Einstein, $W_h^+$ satisfies the second Bianchi identity and it is divergence-free: in particular, since we are assuming that $W_h^+\not\equiv 0$, 
    $W_h^+$ is a non-trivial solution of the first-order elliptic system (see e.g. \cite{Polombo})
    \[
    \mathcal{D}W_h^+=0,
    \]
    where $\mathcal{D}:=d^\nabla+\delta^\nabla$ is the Hodge--Dirac operator on $\Lam^+$-valued 2-forms (see also 
    the discussion before Lemma \ref{lem:weighted-weitzenbock}). 
    Thus, we can apply a result by B\"ar
	\cite[Corollary 1]{Baer1999} and conclude that $N$
	is countably $2$-rectifiable and has finite $\mathcal{H}^2_h$-measure: in particular, $N$ has codimension at least 2.
	
	We now exploit the following consequence of the definition of Hausdorff measure. By compactness of $N$, for every $\delta>0$, there are
	finitely many geodesic balls
	$B_h(p_i,r_i)$, $p_i\in N$, and $0<r_i<\delta$
	such that
	\begin{equation}\label{eq:hausdorff-content-cover}
		N\subset\bigcup_i B_h(p_i,r_i),
		\quad
		\sum_i r_i^2\leq C_N,
	\end{equation}
where $C_N$ is independent of $\delta$. Indeed, finite Hausdorff
measure implies uniformly bounded Hausdorff $\delta$-content for every
$\delta>0$; see, for instance,
\cite[Section 2.1]{EvansGariepy}. Since $W_h^+$ is smooth and vanishes on $N$, $\alpha_h$ is a Lipschitz function in a neighborhood of $N$: hence, for $x$ sufficiently close to $N$,
	\begin{equation*}
		\alpha_h(x)\leq C\,d_h(x,N)
	\end{equation*}
    for some uniform constant $C>0$.
	In particular, choosing $\delta$ sufficiently small, on each ball $B_h(p_i,2r_i)$ we have
	\begin{equation}\label{eq:alpha-ball}
		\alpha_h\leq C r_i.
	\end{equation}
For every Lipschitz function $\varphi$, one has
\begin{equation}\label{eq:conformal-gradient-measure}
|d\varphi|_g^2\,\dd\mu_g
=
\alpha_h^{2/3}|d\varphi|_h^2\,\dd\mu_h
\end{equation}
almost everywhere on $X$.
    We want to construct suitable cutoff functions, adapted to the chosen cover of $N$, in order to apply Proposition 
    \ref{prop:capacity-form}. We choose $0<\delta<\mathrm{inj}_h(M)/2$ and, for every $i$,
	we can define $\theta_i\in C^\infty(M)$ such that 
	\[
	0\leq\theta_i\leq 1, \quad \theta_i=0\quad\text{on }B_h(p_i,r_i),
	\quad
	\theta_i=1\quad\text{on }M\setminus B_h(p_i,2r_i),
	\]
	with
	\[
	|d\theta_i|^2_h\leq C r_i^{-2},
	\]
    where $C$ is a constant independent of $\delta$. Since the functions $\theta_i$ are smooth, the function
\[
\chi_\delta:=\min_i\theta_i
\]
is Lipschitz on $M$. It vanishes on an open neighborhood of $N$, so its
restriction belongs to $\operatorname{Lip}_c(X)$. Moreover,
\[
0\leq\chi_\delta\leq1,
\quad
\chi_\delta=0\quad\text{on }\bigcup_iB_h(p_i,r_i),
\quad
\chi_\delta=1\quad\text{on }M\setminus\bigcup_iB_h(p_i,2r_i),
\]
and
\[
\norm{d\chi_\delta}_h^2
\leq C\sum_i r_i^{-2}\mathbf 1_{B_h(p_i,2r_i)}
\]
almost everywhere on $X$. Let $K\Subset X$. Since
$\bigcup_iB_h(p_i,2r_i)\subset\{x:d_h(x,N)<2\delta\}$, the set $K$ is
disjoint from this union for all sufficiently small $\delta$.
Therefore $\chi_\delta\to1$ uniformly on $K$.
	Finally, we need to verify that
	\begin{equation}\label{eq:g-capacity-zero}
		\int_X |d\chi_\delta|_g^2\,\dd\mu_g\longrightarrow0
		\quad\text{as }\delta\downarrow0.
	\end{equation}
After possibly decreasing $\delta$, we use the standard volume bound for small geodesic balls
	\[
	\operatorname{Vol}_h B_h(p_i,2r_i)\leq C r_i^4;
	\]
    using \eqref{eq:hausdorff-content-cover}, \eqref{eq:alpha-ball} and \eqref{eq:conformal-gradient-measure},
	we obtain
	\begin{align*}
		\int_X |d\chi_\delta|_g^2\,\dd\mu_g
		&=
		\int_X \alpha_h^{2/3}|d\chi_\delta|_h^2\,\dd\mu_h\\
		&\leq
		C\sum_i
		r_i^{2/3}r_i^{-2}
		\operatorname{Vol}_h B_h(p_i,2r_i)
		\leq
		C\sum_i r_i^{8/3}\\
		&\leq
		C\delta^{2/3}\sum_i r_i^2
		\leq
		C\delta^{2/3},
	\end{align*}
	which proves \eqref{eq:g-capacity-zero}. Therefore, the cutoffs $\chi_\delta$ satisfy
    the capacity condition \eqref{eq:abstract-capacity-condition} and Proposition \ref{prop:capacity-form} applies. Thus, after at worst passing to a double cover of $X$,
\[
\nabla^g\omega_1=0,
\quad
\beta_h=\gamma_h=-\frac12\alpha_h
\quad\text{on }X
\]
and $g=\alpha_h^{2/3}h$ is a K\"ahler metric on $X$.
Since $W_h^+\neq0$ on $X$, the eigenvalues of $W_g^+$ are
\[
\alpha=\dfrac{s_g}{6},
\quad
\beta=\gamma=-\dfrac{s_g}{12}.
\]
Since $\alpha>0$ on $X$, it follows that $s_g>0$ on $X$. Now, since $W_h^+= 0$ on $N$ and 
        $\beta_h=\gamma_h$ on $X$, we have that 
        $\#\mathrm{spec}(W_h^+)$, i.e. the number of distinct eigenvalues of $W_h^+$, is at most $2$ on $M$. Therefore, by \cite[Proposition 5(iv)]{Derdzinski} and the fact that
$W_h^+\not\equiv0$, we conclude that $W_h^+$ is nowhere zero, contrary
to $N\neq\varnothing$. Hence $N=\varnothing$, and, after at worst passing to a double cover of $M$, $g$ is K\"ahler with
positive scalar curvature on $M$. Since
$h=\alpha_h^{-2/3}g$, the conformal transformation law for the scalar
curvature gives
\[
\operatorname{Vol}_h(M)s_h
=\int_Ms_h\,\dd\mu_h
=\int_M\left(
\frac23\alpha_h^{-8/3}\norm{\nabla\alpha_h}_g^2
+s_g\alpha_h^{-2/3}
\right)\dd\mu_g>0.
\]
Since $h$ is Einstein, $s_h$ is constant, and hence positive: thus $h$ is a Hermitian Einstein metric with positive scalar curvature, which concludes the proof.
        \end{proof}
	\begin{proof}[Proof of Corollary \ref{cor:det-version}]
		We follow the notation in the proof of Corollary \ref{cor:det-version-hhw}. By assumption and the continuity and strict monotonicity of $\mathcal F$ on the
relevant interval, there is a unique $\rho_0<1$ such that
		\[
		\mathcal{F}(\rho)\geq-c=\mathcal{F}(\rho_0) 
		\]
        at every point where $W_h^+\neq0$,
        where $c<\frac{1}{3\sqrt{6}}$. Hence, by the monotonicity of $\mathcal{F}(\rho)$,
		\[
		\rho\leq\rho_0,
		\]
		i.e.
		\[
		\beta_h\leq\rho_0\alpha_h
		\]
		on the set $\{W_h^+\neq0\}$. Since, on the zero set of $W_h^+$, the same inequality is automatically true, the conclusion follows from Theorem \ref{thm:main}.
	\end{proof}
	
\begin{proof}[Proof of Corollary \ref{cor:nonpositive-scalar-nonexistence}]
Suppose that
\[
\sup_{\{W_h^+\neq0\}}\frac{\beta_h}{\alpha_h}=\rho_0<1.
\]
Then $\beta_h\leq\rho_0\alpha_h$ on $\{W_h^+\neq0\}$, and the same
inequality holds trivially on the zero set. Since
$W_h^+\not\equiv0$, Theorem \ref{thm:main} implies that $h$ has
positive scalar curvature, contradicting the hypothesis.
\end{proof}

\section{K\"ahler--Einstein surfaces}
	\label{sec:ke-holomorphic-pinching}
    We are now ready to prove Theorem \ref{thm:siu-yang-pinching}. Let $(M,h,J)$ be a
K\"ahler--Einstein surface with the orientation induced by $J$. At
each $p\in M$, denote by $H$ the holomorphic sectional curvature,
defined by
\begin{equation}\label{eq:holomorphic-sect-curvature}
H(v)=h(R(v,Jv)Jv,v)
\end{equation}
for every unit vector $v\in T_pM$. We denote its pointwise maximum and
minimum by $H_{\max}$ and $H_{\min}$, respectively, and its
Fubini--Study average over $\mathbb{CP}^1$ by $H_{\mathrm{av}}$.
	Following the notation of Siu and Yang \cite{SiuYang},
	when $H_{\max}>H_{\min}$, we define the pinching
	ratio
	\[
	\Theta=
	\frac{H_{\mathrm{av}}-H_{\min}}{H_{\max}-H_{\min}} .
	\]

It is well known that, with respect to the complex orientation, the
tensor $W_h^-$ of a K\"ahler--Einstein surface is precisely the
\emph{Bochner tensor} $B$; see, for instance,
\cite{Bochner,Tachibana,TricerriVanhecke}. As in the previous sections, let
\[
    \alpha_-\geq\beta_-\geq\gamma_-
\]
denote the eigenvalues of $W_h^-$, and define
\[
    \rho_-:=\frac{\beta_-}{\alpha_-}
\]
at every point where $W_h^-\neq0$. Let $\omega$ be the K\"ahler form, normalized by
$|\omega|_h^2=2$. For every unit vector $v\in T_pM$, set
\[
    \sigma_v:=v^\flat\wedge(Jv)^\flat,
    \quad
    \varphi_v:=\sqrt{2}\left(\sigma_v-\frac12\omega\right)
    \in\Lambda_p^-.
\]
Then $|\varphi_v|_h=1$, and the curvature-operator decomposition of
the K\"ahler--Einstein metric $h$ gives
\[
    H(v)
    =\big\langle\mathcal R_h(\sigma_v),\sigma_v\big\rangle_h
    =\frac{s_h}{6}
     +\frac12\big\langle W_h^-(\varphi_v),\varphi_v\big\rangle_h.
\]
As $v$ varies over the unit tangent sphere, $\varphi_v$ ranges over
the unit sphere in $\Lambda_p^-$. Consequently,
\[
    H_{\max}=\frac{s_h}{6}+\frac{\alpha_-}{2},
    \quad
    H_{\min}=\frac{s_h}{6}+\frac{\gamma_-}{2}.
\]
Furthermore, a well-known result due to Berger 
	\cite{Berger} implies that 
	\[
	H_{\mathrm{av}}=\dfrac{s_h}{6}.
	\]
	Putting everything together, we can rewrite the 
	pinching ratio $\Theta$ as
	\begin{equation}\label{eq:theta-bochner-ratio}
		\Theta
		=
		\frac{-\gamma_-}{\alpha_- -\gamma_-}
		=
		\frac{1+\rho_-}{2+\rho_-}.
	\end{equation}
Since $-1/2\leq\rho_-\leq1$ and
\[
\frac{d\Theta}{d\rho_-}=\frac{1}{(2+\rho_-)^2}>0,
\]
we have
\[
\frac13\leq\Theta\leq\frac23.
\]
Moreover,
\[
\Theta=\frac23
\quad\Longleftrightarrow\quad
\rho_-=1
\quad\Longleftrightarrow\quad
\alpha_-=\beta_-.
\]
	
	\begin{proof}[Proof of Theorem \ref{thm:siu-yang-pinching}]
		Assume first that $W_h^-\not\equiv0$: since $\Theta\geq 1/3$, we immediately see that
		$1/3\leq \theta<2/3$. Then, using the pinching hypothesis $\Theta\leq \theta$ together with \eqref{eq:theta-bochner-ratio}, we obtain the uniform spectral gap
		\[
		\frac{\beta_-}{\alpha_-}
		\leq
		\rho_0<1
		\]
		on $\{W_h^-\neq0\}$, where $\rho_0=(2\theta-1)/(1-\theta)$. After reversing the orientation, $W_h^-$ becomes $W_h^+$. The uniform
bound above then contradicts Corollary
\ref{cor:nonpositive-scalar-nonexistence} unless $W_h^-$ vanishes
identically in the original complex orientation. Hence $W_h^-\equiv0$. If $h$ is Ricci-flat, it is therefore flat, and the Bieberbach theorem
implies that $M$ is finitely covered by a flat complex torus. We may
thus assume that $s_h<0$. For a K\"ahler--Einstein surface,
$W_h^-=0$ is equivalent to the vanishing of the Bochner tensor. A
constant-scalar-curvature K\"ahler metric with vanishing Bochner tensor
has constant holomorphic sectional curvature
\cite{MatsumotoTanno}. Consequently, the universal cover is the
complex hyperbolic plane \cite[Chapter IX]{KobayashiNomizu}, and $M$
is a compact quotient of the complex ball.
\end{proof}

\section{Noncompact Ricci-flat manifolds}
\label{sec:extensions}

In this section, we show how to extend our results to noncompact Ricci-flat manifolds: in particular, under suitable 
volume growth and curvature decay assumptions, we show how
to adapt the argument of Proposition \ref{prop:capacity-form}. 

\begin{proof}[Proof of Theorem \ref{thm:intro-extensions}] Again, we set
    \[
        N=\{W_h^+=0\},
        \quad
        X=M\setminus N .
    \]
    As we observed in Section \ref{sec:closed-einstein},  the condition \eqref{eq:noncompact-uniform-gap} implies that
    $\alpha_h>0$ and that it is simple on $X$: hence, we need to verify the capacity condition
    in Proposition \ref{prop:capacity-form} for
    $g=\alpha_h^{2/3}h$. The restriction of the zero set of $W_h^+$ to each ball $\overline{B_h(p_0,2R)}$ has finite $\mathcal{H}_h^2$-measure by B\"ar's Theorem \cite{Baer1999}: hence, we can apply the
    same argument used in the proof of Theorem \ref{thm:main} in order to
    construct cutoff functions $\eta_{\delta,R}$, $0\leq \eta_{\delta,R}\leq 1$, which vanish near $N\cap \overline{B_h(p_0,2R)}$, are
    equal to 1 away from the zero set and satisfy
    \begin{equation}\label{eq:noncompact-zero-capacity}
        \int_X |d\eta_{\delta,R}|_g^2\,\dd\mu_g
        \leq C_R\delta^{2/3},
    \end{equation}
    where $C_R$ is a constant independent of $\delta$. Moreover, for fixed $R$, $\eta_{\delta,R}\to1$ locally uniformly on $B_h(p_0,2R)\cap X$ as $\delta\to0$. Let $\psi_R$ be a Lipschitz radial cutoff function such that $\psi_R=1$ on
    $B_h(p_0,R)$, $\psi_R=0$ on $M\setminus B_h(p_0,2R)$ and
    $|d\psi_R|_h\leq C R^{-1}$ almost everywhere. Since $h$ is Ricci-flat, $\Rm_h= W_h$, which implies that
    $\alpha_h\leq C|\Rm_h|$: hence, we obtain 
    \begin{align}
        \int_X |d\psi_R|_g^2\,\dd\mu_g
        &=\int_{B_h(p_0,2R)\setminus B_h(p_0,R)} \alpha_h^{2/3}|d\psi_R|_h^2\,\dd\mu_h \nonumber\\
        &\leq C R^{-2-2q/3}\operatorname{Vol}_h(B_h(p_0,2R))\nonumber\\   
        &\leq C R^{\nu-2-2q/3}=o(1)
        \label{eq:noncompact-infinity-capacity}
    \end{align}
    as $R\to\infty$. We now define $\chi_{\delta,R}:=\eta_{\delta,R}\psi_R$: 
    by \eqref{eq:noncompact-zero-capacity} and 
    \eqref{eq:noncompact-infinity-capacity}, we obtain
    \[
        \int_X |d\chi_{\delta,R}|_g^2\,\dd\mu_g
        \leq C_R\delta^{2/3}+o_R(1),
    \]
    where $o_R(1)$ denotes a quantity that goes to $0$ as
    $R\to +\infty$. 
    We choose a sequence of radii $R_j$ such that $R_j\to+\infty$ as $j\to +\infty$. Furthermore, we can find a sequence $\{\delta_j\}$ such that $\delta_j>0$ and
    \[
    \delta_j <\dfrac{1}{R_j}, \quad C_{R_j}\delta_j^{2/3}<\dfrac{1}{R_j}
    \]
    for every $j$. 
    Setting $\chi_j:=\chi_{\delta_j,R_j}\in \operatorname{Lip}_c(X)$, 
    we immediately see that 
    $\chi_j\to1$ locally uniformly in $X$ and
    \[
        \int_X |d\chi_j|_g^2\,\dd\mu_g\longrightarrow0 .
    \]
    Therefore, Proposition \ref{prop:capacity-form} applies, and we can conclude that $g=\alpha_h^{2/3}h$ is a K\"ahler metric
    with positive scalar curvature on $X$. Thus, as in the proof of Theorem \ref{thm:main}, we can use 
    our assumption and \cite[Proposition 5 (iv)]{Derdzinski} to deduce that $N=\varnothing$: hence, after passing, if necessary, to the double cover, $g$ is a K\"{a}hler metric with positive scalar curvature on $M$. 
    This implies that $h$ is a Ricci-flat Hermitian metric; however, $h$ is not a K\"{a}hler metric, since otherwise $W_h^+$ would vanish identically \cite[Proposition 2]{Derdzinski}, and this concludes the proof. 
\end{proof}

Now, we want to adapt these results to the case of gravitational
instantons: following \cite{LiSun}, a gravitational instanton is a complete,
connected, oriented, noncompact Ricci-flat four-manifold with
quadratic curvature decay. Equivalently,
\[
\int_M|\Rm_h|^2\,\dd\mu_h<\infty;
\]
see \cite[Remark 2.3]{LiSun}, where the equivalence is deduced using
the $\varepsilon$-regularity theorem of Cheeger and Tian \cite{CheegerTian}. In our setting, we recall the three types of gravitational instantons, 
which can be distinguished by the spectral properties of $W_h^+$.
Namely, we say that an oriented gravitational instanton is of
\begin{itemize}
    \item \emph{Type I} if $W_h^+= 0$ (in this case, 
    $(M,h)$ is locally hyperk\"{a}hler);
    \item \emph{Type II} if $W_h^+$ has two distinct eigenvalues at every point;
    \item \emph{Type III} if $W_h^+$ has three distinct eigenvalues generically. 
\end{itemize}

Every nonflat gravitational instanton has exactly one end: indeed, if
it had at least two ends, the Cheeger--Gromoll splitting theorem would
give an isometric product $\mathbb R\times N^3$, with $N$ compact, and Ricci-flatness would then force $N$ (and, therefore, $M$) to be flat.

If $(M,h)$ is globally hyperk\"ahler and simply connected, the
asymptotic classification theorem of Sun and Zhang \cite{SunZhang}
implies that its end is of ALE, ALF, ALG, ALH, ALG*, or ALH* type (see
\cite[Section 4]{LiSun} and
\cite{BiquardGauduchon,ChenChenI,GibbonsHawking,GibbonsHawking2,Minerbe});
the ALE cases were classified by Kronheimer
\cite{KronheimerALE,KronheimerTorelli}. Under a noncollapsing volume
assumption, Theorem \ref{thm:intro-extensions} gives the following
characterization of the Type II case.

\begin{corollary}\label{cor:noncollapsed-gravitational-instanton}
    Let $(M,h)$ be an oriented gravitational instanton with Euclidean volume
    growth, i.e. such that, for some $v>0$ and some $p_0\in M$,
    \[
        \operatorname{Vol}_h(B_h(p_0,R))\geq vR^4
        \quad\text{for all }R\gg1.
    \]
    Assume that $W_h^+\not\equiv0$ and that, for some $\rho_0<1$,
    \[
        \beta_h\leq\rho_0\alpha_h
        \quad\text{on }M.
    \]
    Then $M$ has one ALE end, and, after at worst passing to a double cover of $M$, $h$ is conformal to a K\"ahler metric with positive scalar curvature. In particular, $h$ is a Hermitian Ricci-flat and non-K\"ahler gravitational instanton of 
    Type II. 
\end{corollary}

\begin{proof}
    By the classical Bishop--Gromov volume estimate, we obtain that \eqref{eq:noncompact-volume-curvature} holds with $\nu=4$. Since $h$ is Ricci-flat and $|\Rm_h|\in L^2$, the non-collapsing volume assumption allows us to apply the theorem of Bando, Kasue and Nakajima, which implies that the end has an ALE structure of order $4$ \cite{BandoKasueNakajima}; in particular $|\Rm_h|=O(R^{-6})$ on $M\setminus B_h(p_0,R)$. Thus \eqref{eq:noncompact-volume-curvature}
    holds, and the conclusion follows from Theorem \ref{thm:intro-extensions}: the fact that $(M,h)$ is of Type II is an immediate consequence of \eqref{eq:eigenvalues-Kahler}. 
\end{proof}

\begin{remark}
In the noncollapsed ALE case, Corollary
\ref{cor:noncollapsed-gravitational-instanton} can be combined with
known classification results. If $W_h^+=0$, then, up to a quotient,
$(M,h)$ is one of Kronheimer's Type I gravitational instantons
\cite{KronheimerALE,KronheimerTorelli}. If $W_h^+\not\equiv0$ and the
group at infinity satisfies $\Gamma\subset \mathrm{SU}(2)$ in the chosen
orientation, Li's classification
\cite{LiALE,LiGravInstantons} implies that, up to homothety and
isometry, $h$ is the Eguchi--Hanson metric on $T^*\mathbb S^2$ with
the opposite orientation \cite{EguchiHanson}.
\end{remark}

Li \cite{LiALE,LiGravInstantons} proved that collapsed Hermitian
non-K\"ahler gravitational instantons are toric and of ALF or AF type.
Combined with the classification of Biquard and Gauduchon
\cite{BiquardGauduchon}, this gives the following list (see
\cite[Section 5]{LiSun} for a complete description):
\begin{itemize}
    \item Kerr metrics on $\mathbb{S}^2\times\mathbb{R}^2$, with respect to both orientations 
    \cite{GibbonsHawking77};
    \item Chen--Teo metrics on $(\mathbb{S}^2\times\mathbb{R}^2)\#\overline{\mathbb{CP}^2}$, where $\overline{\mathbb{CP}^2}$ is the complex projective space with the opposite orientation \cite{ChenTeo};
    \item Taub-bolt and anti-Taub-bolt metrics \cite{Page-taub};
    \item Taub--NUT metric on $\mathbb{R}^4$ with the opposite orientation, called anti-Taub--NUT metric. 
\end{itemize}
We recall that, with respect to the standard orientation of
$\mathbb{R}^4$, the Taub--NUT metric is hyperk\"ahler and hence is a
Type I gravitational instanton. We now prove the following result.
\begin{theorem}\label{thm:collapsed-gravitational-instanton}
    Let $(M,h)$ be an oriented gravitational instanton with at most cubic
    volume growth, i.e.
    \[
        \operatorname{Vol}_h(B_h(p_0,R))=O(R^3) .
    \]
    Assume that $W_h^+\not\equiv0$ and that, for some $\rho_0<1$,
    \[
        \beta_h\leq\rho_0\alpha_h
        \quad\text{on }M .
    \]
    Then, after at worst passing to a double cover of $M$, $h$ is conformal to a K\"ahler metric with positive scalar curvature. In particular, $h$ is a Hermitian Ricci-flat and non-K\"ahler gravitational instanton of 
    Type II and belongs to one
of the following families:
\begin{itemize}
\item a Kerr metric, with either orientation;
\item a Chen--Teo metric, with its Type II orientation;
\item the Taub--bolt metric, with either orientation
      (Taub--bolt or anti--Taub--bolt);
\item the reversed Taub--NUT metric.
\end{itemize}
\end{theorem}

\begin{proof}
    First, we observe that the construction of the functions 
    $\eta_{\delta,R}$ described in the proof of Theorem 
    \ref{thm:intro-extensions} applies without change in this setting: hence, we just need 
    to verify the capacity condition at infinity, in order to
    apply Proposition \ref{prop:capacity-form}. Let $\psi_R$ be the usual radial
    cutoff function supported in $B_h(p_0,2R)$, equal to 1 on $B_h(p_0,R)$, and with
    $|d\psi_R|_h\leq CR^{-1}$. On $A_R=B_h(p_0,2R)\setminus B_h(p_0,R)$, we can apply H\"older's inequality and use the fact that
    $\alpha_h\leq C|\Rm_h|$ in order to obtain
    \begin{align*}
        \int_X |d\psi_R|_g^2\,\dd\mu_g
        &=\int_{A_R}\alpha_h^{2/3}|d\psi_R|_h^2\,\dd\mu_h \\
        &\leq C R^{-2}
        \left(\int_{A_R}|\Rm_h|^2\,\dd\mu_h\right)^{1/3}
        \operatorname{Vol}_h(A_R)^{2/3}= o_R(1),
    \end{align*}
    where the last equality comes from the fact that 
    $\norm{\Rm_h}\in L^2$ and the volume assumption. As we did in the proof of Theorem \ref{thm:intro-extensions}, we can similarly construct cutoff functions $\chi_j$ which satisfy the hypotheses of Proposition \ref{prop:capacity-form}: therefore, after at worst passing to a double cover of $X$, we conclude that $g$ is a K\"{a}hler metric with positive scalar curvature on $X$ and, using again 
    \cite[Proposition 5(iv)]{Derdzinski}, that $X=M$ as in 
    Theorem \ref{thm:intro-extensions}.
    The classification now follows from the aforementioned classification result of conformally
    K\"ahler gravitational instantons due to Li \cite{LiGravInstantons,LiSun}.
\end{proof}

\begin{remark} \label{rem:cheegertian}
    The proof of Theorem \ref{thm:collapsed-gravitational-instanton} can be deduced directly from Theorem \ref{thm:intro-extensions} combined with the curvature decay proved in \cite{CheegerTian}. Moreover, the proof uses
only the bound
$\operatorname{Vol}_h(B_h(p_0,R))\leq CR^3$ and therefore also yields a
rigidity statement for non-ALF collapsed gravitational instantons.
Indeed, Li's classification \cite{LiGravInstantons} shows that every
Type II gravitational instanton is, up to scaling, of ALE, ALF, or AF
type. Consequently, if $(M,h)$ is of ALG, ALG*, ALH, or ALH* type and
satisfies \eqref{eq:noncompact-uniform-gap}, then it must be of Type I
and hence locally hyperk\"ahler.
\end{remark}

\begin{remark}
    The same capacity proof can also be adapted to isolated quotient-orbifold
    singularities.  In local uniformizing charts the lifted metric is smooth, so
    isolated singular points have zero capacity for $\alpha_h^{2/3}h$; combining
    these local cutoffs with the annular cutoff at infinity gives the
    corresponding orbifold variants.  We do not state these extensions here.
\end{remark}

	\section{Sharpness of the spectral condition}\label{sec:ricci-flat-endpoint-model}
	
	In this section, we provide examples of
	Ricci-flat four-manifolds $(M,h)$ such that,
	at a point $p$, $\alpha_h=\beta_h\neq 0$: this
	shows that the uniform hypothesis on the eigenvalues
	in Theorems \ref{thm:main} and \ref{thm:intro-extensions} is optimal. 
    \subsection{Compact case: K3 surfaces}
	First, we recall that a hyperk\"{a}hler
	manifold $(M,h)$ is a Riemannian manifold endowed with three 
	K\"{a}hler structures $I$, $J$ and $K$, which
	satisfy the quaternionic relations:
	\[
	I^2=J^2=K^2=IJK=-1.
	\]
	In dimension four, according to Kodaira's classification of complex surfaces, the only examples of compact
    hyperk\"{a}hler manifolds are flat tori and K3 surfaces. It is well known that 
    every K3 surface admits a K\"{a}hler metric: the resolution of the
	Calabi conjecture by Yau ensures that, given any K\"{a}hler metric $\hat{h}$ on $M$ with K\"{a}hler form $\hat{\omega}$, there exists
	a unique K\"{a}hler Ricci-flat metric $h$ with K\"{a}hler form 
    $\omega\in [\hat{\omega}]$ \cite{Yau}. Furthermore, by \cite[Proposition 2]{Derdzinski}, we know that all these metrics are also anti-self-dual with respect to the orientation induced by the hyperk\"{a}hler structure (see also \cite{Hitchin}),
	which means that the curvature operator is precisely
	$W_h^-$. 
    Since the metrics constructed by Yau
	are not explicit, in general, no explicit formula for
	the curvature tensor is known. Every K3 surface is spin, and
\[
\hat A(M)=-\frac{\tau(M)}8=2
\]
with respect to the hyperk\"ahler orientation. The Lichnerowicz formula
therefore implies that a K3 surface does not admit a metric of positive scalar
curvature \cite{Lichnerowicz}. Thus, after reversing the hyperk\"ahler
orientation, a nonflat Ricci-flat K3 metric cannot satisfy the
conclusions of Theorems \ref{thm:half-harmonic-simple-top} and
\ref{thm:main}.

We show that, for certain K3 surfaces,
\[
\sup_{\{W_h^+\neq0\}}\rho=1
\]
and that there is a point $p$ with $W_h^+(p)\neq0$ and $\rho(p)=1$.
    
    We want to look at holomorphic isometries of K3 surfaces: 
	Lye showed that, if a hyperk\"{a}hler manifold admits
	a holomorphic isometry of order at least $3$ that
	fixes a subset of positive dimension, then the
	holomorphic sectional curvatures of $I$, $J$ and $K$
	provide some information on the full Riemann tensor
	of the manifold \cite{LyeK3}. 

    In this setting, given a point $p\in M$, we define the holomorphic sectional curvatures $H_{II}$, $H_{JJ}$, $H_{KK}$ and the mixed-type curvatures
    $H_{IJ}$, $H_{IK}$, $H_{JK}$ at $p$ as in \eqref{eq:holomorphic-sect-curvature}: for instance, for every
    $v\in T_pM$ of unit norm,
    \[
    H_{II}(v)=h(R(v,Iv)Iv,v), \quad 
    H_{IJ}(v)=h(R(v,Iv)Jv,v)
    \]
    and so on. It is not hard to show that 
    $H_{IJ}=H_{JI}$, $H_{IK}=H_{KI}$ and $H_{JK}=H_{KJ}$;
    furthermore, since $h$ is Ricci-flat, we have
    \[
    H_{II}+H_{JJ}+H_{KK}=0.
    \]
	Following \cite[Theorem 4.6]{LyeK3}, if $(M,h)$ admits
    a holomorphic isometry of order at least $3$ which fixes a subset $C$ of $M$ with positive dimension, then 
    \begin{equation} \label{eq:holom-sec-curv-K3-properties}
    H_{IJ}(v)=H_{IK}(v)=H_{JK}(v)=0, \quad H_{JJ}(v)=H_{KK}(v)=-\dfrac{1}{2}H_{II}(v)
    \end{equation}
    for every unit tangent vector $v\in T_pC$, $p\in C$:
    in particular, this means that all six curvature quantities are uniquely determined by $H_{II}$ along $C$.
    Assume that $\dim C=2$, i.e. $C$ is a surface inside
    $(M,h)$: since $C$ is fixed by an isometry of $h$, 
    it is a (possibly non-connected) totally geodesic smooth submanifold with respect to the metric $\iota^*(h)$ induced
    by $h$ (see e.g. \cite{Kobayashi_isom}) and, by \cite[Corollary 4.8]{LyeK3}, for every $v\in T_pC$ of unit norm, 
\begin{equation}\label{eq:holomorphicsectLye}
H_{II}(w)=\left[(\alpha^2+\beta^2)^2+(\mu^2+\nu^2)^2
-4(\alpha^2+\beta^2)(\mu^2+\nu^2)\right]\kappa(p),
\end{equation}
where
\[
w=\alpha v+\beta Iv+\mu Jv+\nu Kv,
\quad
\alpha^2+\beta^2+\mu^2+\nu^2=1,
\]
and $\kappa$ is the Gaussian curvature of $(C,\iota^*h)$. Now, let us choose a local orthonormal frame $\{e_i\}$ on an open chart
    of $(M,h)$ such that $e_1$ is tangent to $C$ and
    \begin{align*}
        I(e_1)&=e_2,\quad I(e_3)=e_4;\\
        J(e_1)&=e_3, \quad J(e_2)=-e_4;\\
        K(e_1)&=e_4, \quad K(e_2)=e_3.
    \end{align*}
    Setting $R_{ijkl}=
    h(R(e_i,e_j)e_l, e_k)$, we can exploit \eqref{eq:holomorphicsectLye} with $v=e_1$ and $(\alpha,\beta,\mu,\nu)=(1,0,0,0)$ to obtain,
    on $C$, 
 \begin{align*}
R_{1212}&=H_{II}(e_1)=\kappa, \\ 
R_{1313}&=H_{JJ}(e_1)=-\dfrac{1}{2}H_{II}(e_1)=-\dfrac{1}{2}\kappa=R_{1414};
\end{align*}
since $h$ is Ricci-flat and anti-self-dual, we have
\begin{align*}
R_{3434}&=-R_{1234}=R_{1212}=\kappa, \\
R_{4242}&=-R_{1342}=R_{1313}=R_{1414}=-R_{1423}=R_{2323}=-\dfrac{1}{2}\kappa.
\end{align*}
Furthermore, by \eqref{eq:holom-sec-curv-K3-properties} and
anti-self-duality, we obtain
\[H_{IJ}(e_1)=H_{IK}(e_1)=H_{JK}(e_1)=0 \Longrightarrow
R_{ijik}=0
\]
whenever $i,j,k$ are pairwise distinct. We now reverse the orientation on $M$: since $h$ is Ricci-flat and
self-dual in this orientation, its full curvature operator is $W_h^+$.
A direct computation from the preceding curvature identities gives
\begin{equation}\label{eq:weylK3}
W_h^+=
\begin{pmatrix}
2\kappa & 0 & 0\\
0 & -\kappa & 0\\
0 & 0 & -\kappa
\end{pmatrix}.
\end{equation}
    If we assume that $M$ is compact, then $C$ is a compact surface
    with genus $g$: if $g>1$, the Gauss--Bonnet Theorem implies that there exist points of $C$ where $\kappa<0$, which means
    that $\alpha_h=\beta_h$ at those points. 
    
    We observe that, in their classification of 
    non-symplectic holomorphic automorphisms of order $3$, Artebani and Sarti \cite{ArtebaniSarti} proved the existence of K3 surfaces satisfying our
    hypotheses and such that the fixed point set of
    the automorphism is a surface of genus $g>1$, which proves
    our claim. Furthermore, if we consider the 
    hyperk\"{a}hler orientation (i.e. the one for which 
    $h$ is anti-self-dual), by \eqref{eq:theta-bochner-ratio}, 
    \eqref{eq:weylK3} and the previous discussion, we obtain
    that there exist points of $C$ at which $\Theta=2/3$.
    
    Here we also provide an explicit example
    of an algebraic K3 surface where $\sup_{\{W_h^+\neq 0\}}\rho=1$. 
    We consider a K3 surface $M$ as a branched double
    cover of $\CP^2$, defined in the weighted 
    projective space $\mathbb{P}(1,1,1,3)$: namely, using
    homogeneous coordinates $[x:y:z:w]$, we choose the
    K3 surface described by the equation
    \[
    w^2=x^6+y^6+z^6,
    \]
    which is branched over the sextic curve $x^6+y^6+z^6=0$
    in $\CP^2$. The surface $M$ is a smooth complex submanifold of
$\mathbb{P}(1,1,1,3)$. Indeed, the first-order partial derivatives of the defining polynomial do not have a common zero outside the origin of $\mathbb{C}^4$ and the unique singular point $[0:0:0:1]$ of $\mathbb{P}(1,1,1,3)$ does not belong to $M$.

The weighted projective space $\mathbb{P}(1,1,1,3)$ admits an
orbifold K\"ahler form (see e.g. \cite{RossThomas}), whose restriction to $M$ is
a smooth K\"ahler form $\widehat\omega$. We define
\[
\phi([x:y:z:w])=[e^{2\pi i/3}x:y:z:w]:
\]
it is clear that $\phi$ is a holomorphic automorphism of $M$ and that $\phi^3=\Id$.
We set
\[
\widehat\omega_{\mathrm{inv}}
:=\frac13\left(
\widehat\omega+\phi^*\widehat\omega+(\phi^2)^*\widehat\omega
\right).
\]
By construction, $\widehat\omega_{\mathrm{inv}}$ is a $\phi$-invariant K\"ahler form on $M$ and, therefore, by Yau's theorem, its
K\"ahler class contains a unique Ricci-flat K\"ahler form $\omega$
\cite{Yau}.
Since $\phi$ preserves the cohomology class $[\widehat\omega_{\mathrm{inv}}]$ and $\omega\in [\widehat\omega_{\mathrm{inv}}]$, the form $\phi^*\omega$ lies in the same class and is also Ricci-flat.
Consequently, by uniqueness in Yau's Theorem, $\phi^*\omega=\omega$, so $\phi$ is an isometry of the Ricci-flat K\"ahler metric $h$ associated with $\omega$.

    Now, we look for the fixed point set $\operatorname{Fix}(\phi)$ of $\phi$, 
    i.e. for the points $[x:y:z:w]$ such that
    \[
    (e^{\frac{2\pi i}{3}}x,y,z,w)=(\lambda x, \lambda y, \lambda z,\lambda^3w),
    \]
    for some $\lambda\in\mathbb{C}\setminus\{0\}$. If $\lambda=1$, a point $p=[x:y:z:w]$ lies in $\operatorname{Fix}(\phi)$ if and 
    only if $x=0$, equivalently, if and only if $p\in C$, where
   $C$ is the complex curve defined by the
    equation $w^2=y^6+z^6$. On the other hand, if $\lambda\neq 1$,
    we immediately get that $\lambda=e^{\frac{2\pi i}{3}}$ and that $y=z=0$: using the equation which defines $M$, we 
    obtain the isolated points
    \[
    q_1=[1:0:0:1], \quad q_2=[1:0:0:-1],
    \]
    i.e. $\operatorname{Fix}(\phi)=C\cup \{q_1,q_2\}$. 

    Hence, the fixed point set of $\phi$ contains 
    a submanifold of $(M,h)$ of dimension $2$, which has to
    be totally geodesic. 
    In order to compute the genus of $C$, we observe that $C$ can be seen as the double cover of the complex line $\{x=0\}$, which is a Riemann sphere $\CP^1$, branched over the six roots
    of the polynomial $P(y,z)=y^6+z^6$: by the 
    classical Riemann--Hurwitz Theorem, we have
    \[
    2g_{C}-2=d(2g_{\CP^1}-2)+B=-2d+B,
    \]
    where $d$ is the degree of the cover $C\longrightarrow\CP^1$
    and $B$ is the total ramification. 
Since $d=2$ and $B=6$, we have $g_C=2$. Thus $C$ is a compact
totally geodesic surface in $(M,h)$. Since $\chi(C)<0$, the
Gauss--Bonnet theorem gives a point $p\in C$ with $\kappa(p)<0$.
With respect to the hyperk\"ahler orientation, the preceding
calculation gives
\[
\operatorname{spec}\bigl(W_h^-(p)\bigr)
=\{2\kappa(p),-\kappa(p),-\kappa(p)\}.
\]
After reversing the orientation, this becomes the spectrum of $W_h^+(p)$;
hence $\alpha_h(p)=\beta_h(p)>0$ and $\rho(p)=1$. 

We also obtain examples showing that, without a uniform gap, one
cannot conclude that $W_h^+$ is either identically zero or nowhere
zero. Certain K3 surfaces with Calabi--Yau metrics admit a holomorphic
isometry of order at least $3$ fixing an elliptic curve $C$ (see e.g.
\cite{ArtebaniSarti} and the Kummer surface studied in \cite{LyeK3}).
Since $C$ is totally geodesic and $\chi(C)=0$, the Gauss--Bonnet
theorem implies that its Gaussian curvature vanishes somewhere.
After reversing orientation, \eqref{eq:weylK3} therefore gives a point
$p$ with $W_h^+(p)=0$. On the other hand, $W_h^+\not\equiv0$, since a
K3 surface admits no flat metric.
    
    \subsection{Noncompact case: the Gibbons--Hawking ansatz}
  In the late 1970s, Gibbons and Hawking developed an ansatz for
constructing gravitational instantons
\cite{GibbonsHawking,GibbonsHawking2}. Let $U\subset\mathbb R^3$ be an
open set with $H^2(U;\mathbb R)=0$, endowed with the Euclidean metric
$h_{\mathrm{Eucl}}$, and let $V$ be a positive harmonic function on
$U$. Since $*_{\mathbb R^3}dV$ is closed, there exists
$\omega\in\Omega^1(U)$ such that
\begin{equation}\label{eq:bogomolny}
d\omega=*_{\mathbb R^3}dV.
\end{equation}
Equivalently, $\omega$ satisfies the Bogomolny monopole equation
\[
\nabla^{\mathrm{Eucl}}\times\omega
=\nabla^{\mathrm{Eucl}}V.
\]
On $M=U\times\mathbb{S}^1$, we can define
    a Riemannian metric $h$ which can be written as
    \begin{equation} \label{eq:GibbHawkmetric}
        h=V^{-1}(d\tau+\omega)^2+Vh_{\mathrm{Eucl}},
    \end{equation}
    where $\tau$ is a coordinate along 
    $\mathbb{S}^1$. One has that 
    $h$ is hyperk\"{a}hler (see e.g. \cite[Section 4.1]{LiSun}).
    Hence, with the orientation induced by the hyperk\"{a}hler structure, $h$ is a Ricci-flat, anti-self-dual metric on $M$. 
    
  For simplicity, write $x^1:=x$, $x^2:=y$, and $x^3:=z$. We reverse the orientation on $M$ and choose the
    local orthonormal coframe given by 
    \[
    e^0:=V^{-1/2}(d\tau+\omega), \quad e^i=V^{1/2}dx^i, \mbox{ } i=1,2,3.
    \]
    We recall that $\{dx^i\}$ is a local orthonormal coframe on $\mathbb{R}^3$, with respect to which the Hodge operator on $\mathbb{R}^3$ can be 
    locally expressed as 
    \[
    *_{\mathbb{R}^3}(\eta)=\dfrac{1}{2}\eta_k\varepsilon_{ijk}dx^i\wedge dx^j,
    \]
    for every 1-form $\eta=\eta_idx^i$, where 
    $\varepsilon_{ijk}$ is the Levi-Civita symbol.
    
    We can find the associated Levi-Civita connection forms
    $\{\omega_j^i\}$ by means of Cartan's first structure equations
    \[
    de^\alpha=-\omega_{\beta}^{\alpha}\wedge e^{\beta}:
    \]
    by \eqref{eq:bogomolny}, a straightforward computation yields
    (see also \cite[Lemma 2.2]{LotayOliveira})
    \begin{align} \label{eq:conn-forms-GH}
        \omega_i^0&=-\dfrac{1}{2}V^{-3/2}\left(\partial_iVe^0+\varepsilon_{ijk}
        \partial_j Ve^k\right)\\
        \omega_j^i&=\dfrac{1}{2}V^{-3/2}
        \left(\partial_jVe^i-\partial_i Ve^j-\varepsilon_{ijk}\partial_kV e^0\right), \notag
    \end{align}   
    where $\partial_i V=\frac{\partial}{\partial x^i}V$. 
    It is immediate to observe that
    \begin{equation} \label{eq:relationconnformGH}
        \omega_j^i=\varepsilon_{ijk}\omega_k^0, 
    \end{equation}
    for every $i,j$. Now, we compute the curvature forms using Cartan's second structure equations
    \[
    \dfrac{1}{2}R_{\beta\mu\nu}^\alpha e^{\mu}\wedge e^{\nu}=\Omega_\beta^{\alpha}=d\omega_\beta^{\alpha}+
    \omega_\gamma^\alpha\wedge\omega_\beta^{\gamma}:
    \]
    by \eqref{eq:conn-forms-GH}, \eqref{eq:relationconnformGH} and the basic properties
    of the Levi-Civita symbol, it is immediate to show that
    \[
    \Omega_j^i=\varepsilon_{ijk}\Omega_k^0.
    \]
    Using the harmonicity of $V$ and the 
    properties of $*_{\mathbb{R}^3}$, we obtain
    \[
    \Omega_i^0=E_{ij}\left(e^0\wedge e^j + \dfrac{1}{2}\varepsilon_{jkl}e^k\wedge e^l\right),
    \]
    where 
    \begin{equation} \label{eq:riemannoperGH}
        E_{ij}:=\dfrac{1}{2}\left(V^{-2}\partial_i\partial_jV+
        V^{-3}\norm{\nabla V}^2\delta_{ij}-3V^{-3}\partial_iV\partial_jV\right). 
    \end{equation}
    By our choice of orientation, $(M,h)$ is now 
    self-dual and, since $h$ is Ricci-flat,
    the Riemann curvature operator coincides with $W_h^+$:
    choosing the standard orthogonal basis 
    $\{\eta_1,\eta_2,\eta_3\}$ of $\Lam^+$,
    where
    \begin{align*}
        \eta_j:=e^0\wedge e^j+\dfrac{1}{2}\varepsilon_{jkl}
        e^k\wedge e^l,
    \end{align*}
    we obtain 
    \[
    \Omega_i^0=E_{ij}\eta_j, \quad \Omega_j^i=\varepsilon_{ijk}E_{kl}\eta_l.
    \]
    Hence, by the definition of the self-dual Weyl operator in $\Lam^+_h$, we can immediately compute 
    \begin{equation*}
    W_h^+(\eta_i)=\Omega_i^0+\dfrac{1}{2}\varepsilon_{ijk}\Omega_k^j=
    \Omega_i^0+\dfrac{1}{2}\varepsilon_{ijk}\varepsilon_{jkl}\Omega_l^0=2\Omega_i^0=2
    E_{ij}\eta_j,
    \end{equation*}
    i.e., viewing $W_h^+$ as a matrix, 
    \begin{equation} \label{eq:weyloperGH}
        W_h^+=\dfrac{1}{V^3}\left[
        V\operatorname{Hess}_{\mathbb{R}^3}(V)+
        \norm{\nabla V}^2h_{\mathrm{Eucl}}-3dV\otimes dV\right].
    \end{equation}
  We now recall the multicentered Gibbons--Hawking construction. Fix
distinct points $p_1,\ldots,p_N\in\mathbb R^3$ and set
\[
U=\mathbb R^3\setminus\{p_1,\ldots,p_N\}.
\]
We can find a smooth, positive harmonic function $V$ in 
    $U$ defined as 
    \begin{equation} \label{eq:multi-center-potential}
        V(x):=\varepsilon+\sum_{\alpha=1}^N\dfrac{m_{\alpha}}{\norm{x-p_\alpha}},
    \end{equation}
    where $\varepsilon\geq 0$ and $m_{\alpha}$ is the mass (or charge)
    at the point $p_\alpha$. 

    The product $U\times\mathbb S^1$ is replaced by a principal
$U(1)$-bundle over $U$. Assume that all masses are equal to $m>0$ and
normalize $\tau$ to have period $4\pi m$. The Chern number over a small
sphere surrounding each $p_\alpha$ is then $\pm1$, so the restriction
of the bundle is the Hopf fibration
$\mathbb S^3\to\mathbb S^2$. After adding one point over each center,
where the circle collapses, the Gibbons--Hawking metric extends to a
smooth complete hyperk\"ahler metric; see
\cite[Section 2.3]{GrossWilson}, \cite[Example 2.4]{LotayOliveira},
and \cite[Section 4.1]{LiSun}.

Since the bundle is generally nontrivial, the expression
$d\tau+\omega$ in \eqref{eq:GibbHawkmetric} must be globally replaced
by a connection one-form $\Theta$ with curvature
$\pi^*(*_{\mathbb R^3}dV)$. However, locally, one can choose $\tau$
and $\omega$ so that \eqref{eq:GibbHawkmetric} and all the preceding
computations remain valid. If $\varepsilon>0$, $(M,h)$ becomes an ALF manifold (such as the multi-Taub-NUT space), while if
    $\varepsilon=0$, we obtain an ALE manifold. In both cases 
    $(M,h)$ is a gravitational instanton: as an example, 
    if $\varepsilon=0$ and $N=2$, we can recover the 
    Eguchi--Hanson metric on $T^*\mathbb{S}^2$
    \cite{EguchiHanson} (see e.g. \cite{Dunajski, LotayOliveira} for other examples). We highlight the fact that Anderson, Kronheimer and LeBrun generalized this construction to the case of countably many centers $p_\alpha$ \cite{AndKronLeBrun}.
    
    From now on, we set $N=4$ and assume that 
    $p_1,\ldots,p_4$ are vertices of a regular tetrahedron 
    in $\mathbb{R}^3$: namely, for a constant $R>0$, we define
\begin{align*}
p_1&:=\dfrac{R}{\sqrt3}(1,1,1),
& p_2&:=\dfrac{R}{\sqrt3}(1,-1,-1),\\
p_3&:=\dfrac{R}{\sqrt3}(-1,1,-1),
& p_4&:=\dfrac{R}{\sqrt3}(-1,-1,1),
\end{align*}
    and we assume that $m_\alpha=1$ for every $\alpha$, for the sake of simplicity. 
    One can easily observe that
    $\norm{p_\alpha}=R$ for every $\alpha$ and that,
    by definition of the $p_\alpha$'s, the following
    identities hold:
    \begin{equation} \label{eq:tethraident}
        \sum_{\alpha=1}^4p_{\alpha}=0, \quad 
        \sum_{\alpha=1}^4p_\alpha\otimes p_\alpha=\dfrac{4R^2}{3}I.
    \end{equation}
    First, we observe that there exists at least
    one point $p\in M$ such that $\norm{W_h^+}(p)=0$:
    indeed, we consider the barycenter $b=(0,0,0)$ of the tetrahedron and we can compute
    \begin{align*}
        \restr{\nabla\left(\dfrac{1}{\norm{p_\alpha-x}}\right)}{x=0}=
        \dfrac{p_\alpha}{R^3}, 
        \quad 
        \restr{\operatorname{Hess}_{\mathbb{R}^3}\left(
        \dfrac{1}{\norm{p_\alpha-x}}\right)}{x=0}=
        \dfrac{3p_\alpha\otimes p_\alpha-R^2I}{R^5}.
    \end{align*}
    By \eqref{eq:multi-center-potential} and \eqref{eq:tethraident}, the previous identities imply
    \begin{equation} \label{eq:barycenter-critical}
    \nabla V(0)=
    \operatorname{Hess}_{\mathbb{R}^3}V(0)=
    0,
    \end{equation}
    which immediately shows that $W_h^+= 0$ for every point $p\in\pi^{-1}(b)$, 
    by \eqref{eq:weyloperGH}. 

    Now, let 
    \[
    r:=\{t\nu: t\in\mathbb{R}\}, \quad 
    \nu:=\dfrac{1}{\sqrt{3}}(1,1,1)
    \]
    be the line through the barycenter $b$
    and $p_1$, and let $x=t\nu$ with $t<R$.
    Let $C_3=\{e, \phi, \phi^2\}$ be the cyclic group generated by $\phi$, the rotation of angle $2\pi/3$ around the axis $r$: it is clear that $r$ is exactly the fixed locus
    of every isometry in $C_3$.
 
    We choose a small, $C_3$-invariant open ball $B\subset U$
    centered at $b$ such that $\pi^{-1}(B)\cong B\times \mathbb{S}^1$ and $p_\alpha\not\in B$ for every $\alpha$. Since $V$ is $C_3$-invariant and the action preserves the Euclidean orientation, $*_{\mathbb{R}^3}dV$ is also $C_3$-invariant. It is important to observe that \eqref{eq:bogomolny} does not admit a unique 
    solution $\omega$: by defining 
    \[
    \omega_{\mathrm{inv}}
    :=\frac13\left(\omega+\phi^*\omega+(\phi^2)^*\omega\right),
    \]
    we obtain $d\omega_{\mathrm{inv}}=*_{\mathbb{R}^3}dV$ and that $\omega_{\mathrm{inv}}$ is $C_3$-invariant. 
    Renaming this new form $\omega$, we can extend $\phi$ to a local isometry $\widetilde{\phi}$ of the total space of the $U(1)$-bundle over $B$, i.e. we can define 
    \begin{equation} \label{eq:local-isometry-GH}
    \widetilde{\phi}(x,\tau)=(\phi(x),\tau),
    \end{equation}
    which obviously generates a cyclic group of local isometries of
    the total space, which we call again $C_3$. 

Now, we fix $p\in\pi^{-1}(r\cap B)$: since $\widetilde\phi(p)=p$, its
differential induces an isometry $\Phi$ of $T_p^*M$. Labeling the
coordinates cyclically, its action is
\begin{align*}
\Phi(e^0)&=e^0,
&\Phi(e^1)&=e^2,
&\Phi(e^2)&=e^3,
&\Phi(e^3)&=e^1,\\
\Phi^2(e^0)&=e^0,
&\Phi^2(e^1)&=e^3,
&\Phi^2(e^2)&=e^1,
&\Phi^2(e^3)&=e^2:
\end{align*}
we point out that the form $e^0$ is fixed by all isometries since both $V$ and $\omega$ are
$C_3$-invariant. Therefore, the natural action on $\Lambda_p^+$ induced by $\Phi$ satisfies
\begin{align*}
\Phi(\eta_1)&=\eta_2,
&\Phi(\eta_2)&=\eta_3,
&\Phi(\eta_3)&=\eta_1,\\
\Phi^2(\eta_1)&=\eta_3,
&\Phi^2(\eta_2)&=\eta_1,
&\Phi^2(\eta_3)&=\eta_2:
\end{align*}
hence, $C_3'=\{\Id,\Phi,\Phi^2\}$ acts by orientation-preserving
isometries on $(\Lambda_h^+)_p$.

    We can visualize $\Phi$ and $\Phi^2$ using their representative
    matrices with respect to the basis $\{\eta_1,\eta_2,\eta_3\}$ of $(\Lam^+)_p$ as 
    \[
    \Phi \longleftrightarrow \begin{pmatrix}
        0 & 0 & 1\\
        1 & 0 & 0\\
        0 & 1 & 0
    \end{pmatrix}, \quad 
    \Phi^2\longleftrightarrow \begin{pmatrix}
        0 & 1 & 0\\
        0 & 0 & 1\\
        1 & 0 & 0
    \end{pmatrix}.
    \]
    Now, since the elements of $C_3'$ are isometries, the self-dual
    Weyl operator $W_h^+$ commutes with all of them: 
    since $W_h^+$ is symmetric and trace-free, it is easy to see that, for
    some $\lambda\in\mathbb{R}$,
    \begin{equation} \label{eq:weyl-oper-GH-axis}
    W_h^+=\begin{pmatrix}
        0 & \lambda & \lambda\\
        \lambda & 0 & \lambda\\
        \lambda & \lambda & 0
    \end{pmatrix}
    \end{equation}
    at $p$. Its characteristic polynomial is
\[
\det(W_h^+-s\Id)=-(s-2\lambda)(s+\lambda)^2.
\]
Thus, its eigenvalues are $2\lambda,-\lambda,-\lambda$. In particular,
at every point of $\pi^{-1}(r\cap B)$ where $W_h^+$ is nonzero, it has
exactly two distinct eigenvalues. Our aim is to show that the ratio $\rho=\beta_h/\alpha_h$ 
    does not admit a continuous extension on $\pi^{-1}(b)$. In order to do so, we exploit the conclusions drawn \emph{via} the properties of the $C_3$-symmetry and 
    consider the expansion of $V$ around $b$: indeed,
    we will show that $W_h^+$ does not admit other zeros on the fibers of points on the axis $t\nu$ close to
    $b$, and therefore, we prove that, for $t\to 0$, the limit of $\rho(t):=\rho(t\nu)$ does not exist.
    First, we can easily compute the following
    \[
    \restr{\nabla^3\left(\dfrac{1}{\norm{p_\alpha-x}}\right)}{x=0}=
    \dfrac{3\left(5p_\alpha\otimes p_\alpha\otimes p_\alpha-
    3R^2 h_{\mathrm{Eucl}}\odot p_\alpha^{\flat}\right)}{R^7},
    \]
    where $\nabla^3$ denotes the third covariant derivative with respect to $h_{\mathrm{Eucl}}$ and $h_{\mathrm{Eucl}}\odot p_\alpha^{\flat}$ denotes the symmetric part
    of $h_{\mathrm{Eucl}}\otimes p_\alpha^{\flat}$; then, 
    by definition of the $p_\alpha$'s, we can compute,
    for every $x=(x_1,x_2,x_3)\in\mathbb{R}^3$,
    \begin{align} \label{eq:tetraidentcubic}
    \sum_{\alpha=1}^4\langle p_\alpha,x\rangle^3=&
    \dfrac{R^3}{3\sqrt{3}}\left[(x_1+x_2+x_3)^3+(x_1-x_2-x_3)^3\right.\notag \\
    &+\left.(-x_1+x_2-x_3)^3+(-x_1-x_2+x_3)^3\right]=\dfrac{8R^3}{\sqrt{3}}x_1x_2x_3. 
    \end{align}
    We consider
    the following expansion of $1/\norm{p_\alpha-x}$ 
    centered at $b$:
    \begin{align*}
    \dfrac{1}{\norm{p_\alpha-x}}=&
    \dfrac{1}{R}+\dfrac{\langle p_\alpha, x\rangle}{R^3}+
    \dfrac{3\langle p_\alpha,x\rangle^2-R^2\norm{x}^2}{2R^5}\\&+
    \dfrac{5\langle p_\alpha, x\rangle^3-3R^2\norm{x}^2\langle p_\alpha,x\rangle}{2R^7}+O(\norm{x}^4). 
    \end{align*}
    By \eqref{eq:multi-center-potential}, \eqref{eq:tethraident} and \eqref{eq:tetraidentcubic},
    since $V(0)=\varepsilon+4/R$, we obtain the following
    \[
    V(x)=\varepsilon+\dfrac{4}{R}+\dfrac{20}{\sqrt{3}R^4}x_1x_2x_3+
    O(\norm{x}^4) 
    \]
    around $b$. By \eqref{eq:weyloperGH}, it is clear
    that, for every $p\in\pi^{-1}(b')$, where $b'$ is close to $b$,
    \begin{equation} \label{eq:weyl-oper-GH-asymptotic}
    W_h^+=\dfrac{20}{\sqrt{3}R^4V(0)^2}
    \begin{pmatrix}
        0 & x_3 & x_2\\
        x_3 & 0 & x_1\\
        x_2 & x_1 & 0    
        \end{pmatrix}
        +O(\norm{x}^2).
    \end{equation}
    Indeed, the product of the first derivatives of $V$ is 
    $O(|x|^4)$, so the leading term in \eqref{eq:weyloperGH}
    becomes $V^{-2}\operatorname{Hess}_{\mathbb{R}^3}(V)$. 

    Now, we restrict ourselves to $x=t\nu$ for some $t$, which means that
    \[
    \lambda(t)=\dfrac{20}{3R^4V(0)^2}t+O(t^2)
    \]
    in \eqref{eq:weyl-oper-GH-axis}. Putting $c=20/(3R^4V(0)^2)$,
    the eigenvalues 
    of $W_h^+$ are 
    \[
    2ct+O(t^2), \quad -ct+O(t^2), \quad -ct+O(t^2):
    \]
    by our previous discussion about the action of $C_3'$, 
    the two equal eigenvalues have exactly the same error term
    for every $t$. Consequently,
\[
\rho(t)=
\begin{cases}
1, & t<0,\\[2mm]
-\dfrac12, & t>0,
\end{cases}
\quad\text{for }0<|t|\ll1.
\]
In particular, 
    \[
    \lim_{t\to 0^+} \rho(t)=-\dfrac{1}{2}, \quad
    \lim_{t\to 0^-}\rho(t)=1, 
    \]
    hence there is a degeneration phenomenon for $\rho$ on $\pi^{-1}(b)$:
    this means that there exists no constant $\rho_0<1$ such that
    $\beta_h\leq\rho_0\alpha_h$ in $M$. 

This shows that the conclusion of Theorem
\ref{thm:intro-extensions} can fail if the spectral hypothesis is
omitted. The metric $h$ is a nonflat gravitational instanton with an
ALE or ALF end and satisfies the volume-growth and curvature-decay
assumptions with $(\nu,q)=(4,4)$ in the ALE case and $(\nu,q)=(3,3)$
in the ALF case. Nevertheless, a K\"ahler metric of positive scalar
curvature satisfies $\det(W_g^+)>0$ everywhere
\cite{Derdzinski}. By conformal invariance, a K\"ahler representative
of this type in $[h]$ would imply $\det(W_h^+)>0$, contradicting the
vanishing of $W_h^+$ in $\pi^{-1}(b)$.

\begin{remark}
The examples above show that the spectral hypotheses in Theorems
\ref{thm:main} and \ref{thm:intro-extensions} are sharp. We have
not produced a compact example that exhibits the same loss of continuity
of $\rho$ near a zero of $W_h^+$ as in the Gibbons--Hawking example.
Despite the lack of explicit formulas for Calabi--Yau metrics on K3
surfaces, they may provide such examples.
\end{remark}


\begin{thebibliography}{99}
		
        \bibitem{AndKronLeBrun}
        M. T. Anderson, P. B. Kronheimer, C. LeBrun, \emph{Complete Ricci-flat K\"ahler manifolds of infinite topological type}, Comm. Math. Phys. {\bf 125} (1989), no. 4, 637--642.
		
		\bibitem{ArtebaniSarti}
		M. Artebani, A. Sarti,
		\emph{Non-symplectic automorphisms of order 3 on K3 surfaces},
		Math. Ann. \textbf{342} (2008), no. 4, 903--921.
		
		\bibitem{BandoKasueNakajima}
		S. Bando, A. Kasue, H. Nakajima,
		\emph{On a construction of coordinates at infinity on manifolds with fast
			curvature decay and maximal volume growth},
		Invent. Math. \textbf{97} (1989), no. 2, 313--349.

        \bibitem{Baer1999}
		C. B\"ar,
		\emph{Zero sets of solutions to semilinear elliptic systems of first order},
		Invent. Math. \textbf{138} (1999), no. 1, 183--202.
		
		\bibitem{Berger}
		M. Berger, Sur les vari\'et\'es d'Einstein compactes, in {\it Comptes Rendus de la IIIe R\'eunion du Groupement des Math\'ematiciens d'Expression Latine (Namur, 1965)}, pp. 35--55, Librairie Universitaire, Louvain.
		
		\bibitem{Besse}
		A. L. Besse, {\it Einstein manifolds}, reprint of the 1987 edition, 
		Classics in Mathematics, Springer, Berlin, 2008.
		
		\bibitem{BiquardGauduchon}
		O. Biquard, P. Gauduchon,
		\emph{On toric Hermitian ALF gravitational instantons},
		Comm. Math. Phys. \textbf{399} (2023), no. 1, 389--422.
		
		\bibitem{BiquardGauduchonLeBrun}
		O. Biquard, P. Gauduchon, and C. LeBrun,
		\emph{Gravitational instantons, Weyl curvature, and conformally K\"ahler geometry},
		Int. Math. Res. Not. IMRN \textbf{2024}, no. 20, 13295--13311.

        \bibitem{Bochner}
        S. Bochner, \emph{Curvature and Betti numbers. II}, Ann. of Math. (2) {\bf 50} (1949), 77--93.

        \bibitem{Bourguignon}
        J.-P. Bourguignon, Les vari\'et\'es de dimension $4$\ \`a{} signature non nulle dont la courbure est harmonique sont d'Einstein, Invent. Math. {\bf 63} (1981), no. 2, 263--286.
		
		\bibitem{CatinoMastroliaBook}
		G. Catino, P. Mastrolia, \emph{A perspective on canonical Riemannian metrics}, Progress in Mathematics, vol. 336, Birkh\"auser/Springer, Cham, 2020.
		
        \bibitem{CheegerTian}
		J. Cheeger, G. Tian,
		\emph{Curvature and injectivity radius estimates for Einstein 4-manifolds},
		J. Amer. Math. Soc. \textbf{19} (2006), no. 2, 487--525.

		\bibitem{ChenChenI}
		G. Chen, X. Chen,
		\emph{Gravitational instantons with faster than quadratic curvature decay (I)},
		Acta Math. \textbf{227} (2021), no. 2, 263--307.

        \bibitem{ChenLeBrunWeber}
		X. X. Chen, C. LeBrun, B. Weber,
		\emph{On conformally K\"ahler, Einstein manifolds},
		J. Amer. Math. Soc. \textbf{21} (2008), no. 4, 1137--1168.

        \bibitem{ChenTeo}
        Y. Chen, E. Teo, \emph{A new AF gravitational instanton}, Phys. Lett. B {\bf 703} (2011), no.~3, 359--362.
		
		\bibitem{Derdzinski}
		A. Derdzi\'nski,
		\emph{Self-dual K\"ahler manifolds and Einstein manifolds of dimension four},
		Compositio Math. \textbf{49} (1983), no. 3, 405--433.

        \bibitem{Dunajski}
        M. Dunajski, \emph{Gravitational instantons, old and new}, Acta Phys. Polon. B {\bf 55} (2024), no. 12, Paper No. A3, 18 pp.
		
		\bibitem{EguchiHanson}
		T. Eguchi, A. J. Hanson,
		\emph{Self-dual solutions to Euclidean gravity},
		Ann. Physics \textbf{120} (1979), no. 1, 82--106.

        \bibitem{EvansGariepy}
		L. C. Evans, R. F. Gariepy,
		\emph{Measure theory and fine properties of functions},
		revised ed., Textbooks in Mathematics, CRC Press, Boca Raton, FL, 2015.

        \bibitem{FriedrichKurke}
        T. Friedrich, H. Kurke, \emph{Compact four-dimensional self-dual Einstein manifolds with positive scalar curvature}, Math. Nachr. {\bf 106} (1982), 271--299.

        \bibitem{GibbonsHawking77}
        G. W. Gibbons, S. W. Hawking, \emph{Action integrals and partition functions in quantum gravity}, Phys. Rev. D {\bf 15} (1977), 2752--2756.
        
         \bibitem{GibbonsHawking}
        G. W. Gibbons, S. W. Hawking, \emph{Gravitational multi-instantons}, Phys. Lett. B {\bf 78} (1978), no. 4, 430--432.

        \bibitem{GibbonsHawking2}
        G. W. Gibbons, S. W. Hawking, \emph{Classification of gravitational instanton symmetries}, Comm. Math. Phys. {\bf 66} (1979), no. 3, 291--310.

        \bibitem{GrossWilson}
        M. Gross, P. M. H. Wilson, \emph{Large complex structure limits of $K3$ surfaces}, J. Differential Geom. {\bf 55} (2000), no. 3, 475--546.
		
		\bibitem{Guan2017}
		D. Guan,
		\emph{On bisectional nonpositively curved compact K\"ahler--Einstein surfaces},
		Pacific J. Math. \textbf{288} (2017), no. 2, 343--353.

        \bibitem{GuenanciaHamenstadt}
		H. Guenancia, U. Hamenst\"adt,
		\emph{K\"ahler--Einstein metrics of negative curvature},
		preprint, arXiv:2503.02838v3, 2026.

        \bibitem{Hitchin}
		N. J. Hitchin,
		\emph{Compact four-dimensional Einstein manifolds},
		J. Differential Geom. \textbf{9} (1974), 435--441.

        \bibitem{HitchinKahler}
        N. J. Hitchin, \emph{K\"ahlerian twistor spaces}, Proc. London Math. Soc. (3) {\bf 43} (1981), no. 1, 133--150.

        \bibitem{Kobayashi_isom}
        S. Kobayashi, \emph{Fixed points of isometries}, Nagoya Math. J. {\bf 13} (1958), 63--68.

        \bibitem{KobayashiNomizu}
        S. Kobayashi, K. Nomizu, {\it Foundations of differential geometry. Vol. II}, reprint of the 1969 original, Wiley Classics Library, Wiley-Interscience, Wiley, New York, 1996. 

        \bibitem{KronheimerALE}
		P. B. Kronheimer,
		\emph{The construction of ALE spaces as hyper-K\"ahler quotients},
		J. Differential Geom. \textbf{29} (1989), no. 3, 665--683.
		
		\bibitem{KronheimerTorelli}
		P. B. Kronheimer,
		\emph{A Torelli-type theorem for gravitational instantons},
		J. Differential Geom. \textbf{29} (1989), no. 3, 685--697.
		
		\bibitem{LeBrun2012}
		C. LeBrun,
		\emph{On Einstein, Hermitian 4-manifolds},
		J. Differential Geom. \textbf{90} (2012), no. 2, 277--302.
		
		\bibitem{LeBrun2013}
		C. LeBrun,
		\emph{Einstein manifolds and extremal K\"ahler metrics},
		J. Reine Angew. Math. \textbf{678} (2013), 69--94.
		
		\bibitem{LeBrun2015}
		C. LeBrun, \emph{Einstein metrics, harmonic forms, and symplectic four-manifolds}, Ann. Global Anal. Geom. {\bf 48} (2015), no. 1, 75--85.
		
		\bibitem{LeBrun2021}
		C. LeBrun,
		\emph{Einstein manifolds, self-dual Weyl curvature, and conformally K\"ahler geometry},
		Math. Res. Lett. \textbf{28} (2021), no. 1, 127--144.
		
		\bibitem{LiALE}
		M. Li,
		\emph{On 4-dimensional Ricci-flat ALE manifolds},
		preprint, arXiv:2304.01609, 2023.
		
		\bibitem{LiGravInstantons}
		M. Li,
		\emph{Classification results for conformally K\"ahler gravitational instantons},
		preprint, arXiv:2310.13197, 2023.

        \bibitem{LiSun}
        M. Li, S. Sun,
        \emph{Geometry of gravitational instantons},
        to appear in \emph{Surveys in Differential Geometry},
        arXiv:2606.10443, 2026.

        \bibitem{Lichnerowicz}
        A. Lichnerowicz, \emph{Spineurs harmoniques}, C. R. Acad. Sci. Paris {\bf 257} (1963), 7--9.

        \bibitem{LotayOliveira}
        J. D. Lotay, G. Oliveira, \emph{Special Lagrangians, Lagrangian mean curvature flow and the Gibbons-Hawking ansatz}, J. Differential Geom. {\bf 126} (2024), no. 3, 1121--1184.
		
		\bibitem{LyeK3}
		J. O. Lye,
		\emph{Geodesics on a K3 surface near the orbifold limit},
		Ann. Global Anal. Geom. \textbf{63} (2023), Paper No. 20.

        \bibitem{MatsumotoTanno}
        M. Matsumoto, S. Tanno, \emph{K\"ahlerian spaces with parallel or vanishing Bochner curvature tensor}, Tensor (N.S.) {\bf 27} (1973), 291--294.
		
		\bibitem{Minerbe}
		V. Minerbe,
		\emph{On the asymptotic geometry of gravitational instantons},
		Ann. Sci. \`Ecole Norm. Sup. (4) \textbf{43} (2010), no. 6, 883--924.
		
		\bibitem{OdakaSpottiSun}
		Y. Odaka, C. Spotti, S. Sun,
		\emph{Compact moduli spaces of del Pezzo surfaces and K\"ahler--Einstein metrics},
		J. Differential Geom. \textbf{102} (2016), no. 1, 127--172.

        \bibitem{Page-taub}
        D. N. Page, \emph{Taub-NUT instanton with an horizon}, Phys. Lett. B {\bf 78} (1978), no. 2, 249--251.
        
		\bibitem{Page}
		D. N. Page,
		\emph{A compact rotating gravitational instanton},
		Phys. Lett. B \textbf{79} (1979), 235--238.
		
		\bibitem{PenroseRindler2}
		R. Penrose, W. Rindler, \emph{Spinors and space-time. Vol. 2}, Cambridge Monographs on Mathematical Physics, Cambridge University Press, Cambridge, 1986.
		
        \bibitem{Polombo}
        A. Polombo,
        \emph{De nouvelles formules de Weitzenb\"ock pour des endomorphismes
        harmoniques. Applications g\'eom\'etriques},
        Ann. Sci. \`Ecole Norm. Sup. (4) \textbf{25} (1992), no. 4, 393--428.

        \bibitem{RossThomas}
        J. Ross, R. P. Thomas, {\it Weighted projective embeddings, stability of orbifolds, and constant scalar curvature K\"ahler metrics}, J. Differential Geom. {\bf 88} (2011), no. 1, 109--159.
		
		\bibitem{SiuYang}
		Y.-T. Siu, P. Yang,
		\emph{Compact K\"ahler--Einstein surfaces of nonpositive bisectional curvature},
		Invent. Math. \textbf{64} (1981), no. 3, 471--487.

        \bibitem{SunZhang}
		S. Sun, R. Zhang,
		\emph{Collapsing geometry of hyperk\"ahler four-manifolds and applications},
		Acta Math. \textbf{232} (2024), 325--424.

        \bibitem{Tachibana}
        S. Tachibana, \emph{On the Bochner curvature tensor}, Natur. Sci. Rep. Ochanomizu Univ. {\bf 18} (1967), 15--19.
		
		\bibitem{Tian}
		G. Tian,
		\emph{On Calabi's conjecture for complex surfaces with positive first Chern class},
		Invent. Math. \textbf{101} (1990), no. 1, 101--172.

        \bibitem{TricerriVanhecke}
        F. Tricerri, L. Vanhecke, \emph{Curvature tensors on almost Hermitian manifolds}, Trans. Amer. Math. Soc. {\bf 267} (1981), no. 2, 365--397. 
		
		\bibitem{Wu}
		P. Wu,
		\emph{Einstein four-manifolds with self-dual Weyl curvature of nonnegative determinant},
		Int. Math. Res. Not. IMRN \textbf{2021}, no. 2, 1043--1054.
		
		\bibitem{Yau}
		S.-T. Yau,
		\emph{On the Ricci curvature of a compact K\"ahler manifold and the complex
			Monge--Amp\`ere equation. I},
		Comm. Pure Appl. Math. \textbf{31} (1978), no. 3, 339--411.
		
	\end{thebibliography}
\end{document}